\documentclass{article}

\usepackage{amsmath,amssymb,amsfonts,amsthm,amsrefs}

\usepackage{multirow}

\usepackage{enumitem}

\usepackage{subcaption}

\usepackage{mathrsfs}

\usepackage{tikz}
\usetikzlibrary{decorations.markings}
\usetikzlibrary{shapes.misc}

\tikzset{cross/.style={cross out, draw=red, minimum size=10*(#1-\pgflinewidth), inner sep=0pt, outer sep=0pt},
cross/.default={1pt}}

\usepackage{xypic}

\usepackage{color}
\usepackage{xcolor}
\colorlet{lightgrey}{white!93!black}

\usepackage{exercise}

\usepackage{array}
\newcolumntype{M}[1]{>{\centering\arraybackslash}m{#1}}

\setitemize{itemsep=0em}

\usepackage{aliascnt}

\newcommand{\thdef}[2]{
	\newaliascnt{#1}{theorem}  
	\newtheorem{#1}[#1]{#2}
	\aliascntresetthe{#1}  
	\newtheorem*{#1*}{#2}
	\expandafter\newcommand\expandafter{\csname #1autorefname\endcsname}{#2}
}

\newtheorem{theorem}{Theorem}[section]
\thdef{proposition}{Proposition}
\thdef{corollary}{Corollary}
\thdef{lemma}{Lemma}
\thdef{construction}{Construction}
\thdef{conjecture}{Conjecture}
\thdef{problem}{Problem}

\theoremstyle{definition}
\thdef{definition}{Definition}

\theoremstyle{remark}
\thdef{remark}{Remark}
\thdef{example}{Example}
\thdef{exercise}{Exercise}

\usepackage{hyperref}
\definecolor{tocolor}{rgb}{.1,.1,.1}
\definecolor{urlcolor}{rgb}{.2,.2,.6}
\definecolor{linkcolor}{rgb}{.1,.1,.5}
\definecolor{citecolor}{rgb}{.4,.2,.1}
\hypersetup{backref=true, colorlinks=true, urlcolor=urlcolor, linkcolor=linkcolor, citecolor=citecolor}

\newcommand{\CC}{\mathbb{C}}

\newcommand{\ZZ}{\mathbb{Z}}
\newcommand{\QQ}{\mathbb{Q}}

\newcommand{\Res}{\mathrm{Res}}
\newcommand{\Bires}{\mathrm{BiRes}}

\newcommand{\cvf}[1]{\partial_{#1}}

\newcommand{\pderiv}[2]{\frac{\partial #1}{\partial #2}}
\newcommand{\lie}{\mathscr{L}}

\newcommand{\rbrac}[1]{\left(#1\right)}

\newcommand{\abrac}[1]{\left\langle#1\right\rangle}
\newcommand{\set}[2]{\left\{#1\,\middle|\,#2\right\}}

\newcommand{\mapdef}[5]{
	\begin{array}{ccccc}
	#1 &:& #2 &\to& #3 \\
		&&  #4 &\mapsto& #5
	\end{array}
}

\newcommand{\defn}[1]{\textbf{\emph{#1}}}

\newcommand{\cO}{\mathcal{O}} % holomorphic
\newcommand{\acan}{\mathcal{K}^{-1}} % anticanonical
\newcommand{\can}{\mathcal{K}} % canonical
\newcommand{\cI}{\mathcal{I}} % ideal
\newcommand{\sL}{\mathcal{L}} % line bundle
\newcommand{\cE}{\mathcal{E}} % line bundle
\newcommand{\cJ}{\mathcal{J}} % jet
\newcommand{\fm}{\mathfrak{m}} % max ideal
\newcommand{\cEnd}{\mathcal{E}nd}

\newcommand{\tshf}[1]{\mathcal{T}_{#1}} % tangent sheaf
\newcommand{\cotshf}[1]{\mathcal{T}^\vee_{#1}} % tangent sheaf
\newcommand{\forms}[1][\bullet]{\Omega^{#1}} % holomorphic forms
\newcommand{\der}[2][\bullet]{\mathscr{X}^{#1}_{#2}} % multiderivations
\newcommand{\hook}[1]{\iota_{#1}} % interior product

\newcommand{\spc}[1]{\mathsf{#1}}
\newcommand{\X}{\spc{X}}
\newcommand{\tX}{\widetilde{\spc{X}}}

\newcommand{\B}{\spc{B}}
\newcommand{\E}{\spc{E}}
\newcommand{\F}{\spc{F}}
\newcommand{\bL}{\spc{L}}
\newcommand{\Y}{\spc{Y}}
\newcommand{\Z}{\spc{Z}}
\newcommand{\W}{\spc{W}}
\newcommand{\Wc}{\spc{W^\circ}}
\newcommand{\G}{\spc{G}}
\newcommand{\U}{\spc{U}}

\newcommand{\D}{\spc{D}}
\newcommand{\T}{\spc{T}}

\newcommand{\R}{\spc{R}}

\newcommand{\Dgn}[1]{\spc{Dgn}_{#1}}
\newcommand{\Zeros}{\spc{Zeros}}
\newcommand{\PP}{\mathbb{P}}

\newcommand{\bS}{\spc{S}}

\newcommand{\Aut}{\spc{Aut}}

\newcommand{\Hilb}{\spc{Hilb}}
\newcommand{\Pois}{\spc{Pois}}
\newcommand{\LogSymp}{\spc{LogSymp}}
\newcommand{\cohlgy}[2][\bullet]{\spc{H}^{#1}\!\rbrac{#2}}
\newcommand{\hlgy}[2][\bullet]{\spc{H}_{#1}(#2)}

\newcommand{\tu}{\widetilde{u}}
\newcommand{\tv}{\widetilde{v}}

\newcommand{\g}{\mathfrak{g}}
\newcommand{\ft}{\mathfrak{t}}
\newcommand{\sln}[1]{\mathfrak{sl}\rbrac{#1}}
\newcommand{\gln}[1]{\mathfrak{gl}\rbrac{#1}}
\newcommand{\aff}[1]{\mathfrak{aff}\rbrac{#1}}
\newcommand{\spn}[1]{\mathfrak{sp}\rbrac{#1}}

\newcommand{\SL}[1]{\spc{SL}\!\rbrac{#1}}
\newcommand{\SU}[1]{\spc{SU}\!\rbrac{#1}}
\newcommand{\Gr}[1]{\spc{Gr}\!\rbrac{#1}}

\newcommand{\red}{\mathrm{red}}
\newcommand{\sing}{\mathrm{sing}}
\newcommand{\codim}{\mathrm{codim}}

\newcommand{\rank}{\textrm{rank}}

\DeclareMathOperator{\ch}{\mathrm{ch}}
\newcommand{\classseven}{$\textrm{VII}_0$}

\newcommand{\cubicscale}{0.25}

\usepackage{titling}
\begin{document}

\title{Constructions and classifications of projective Poisson varieties}

\author{Brent Pym\thanks{School of Mathematics, University of Edinburgh, \href{mailto:brent.pym@ed.ac.uk}{brent.pym@ed.ac.uk}}}

\maketitle
\vspace{-0.5cm}
\begin{abstract}
This paper is intended both an introduction to the algebraic geometry of holomorphic Poisson brackets, and as a survey of results on the classification of projective Poisson manifolds that have been obtained in the past twenty years.  It is based on the lecture series delivered by the author at the Poisson 2016 Summer School in Geneva.

The paper begins with a detailed treatment of Poisson surfaces, including adjunction, ruled surfaces and blowups, and leading to a statement of the full birational classification.  We then describe several constructions of Poisson threefolds, outlining the classification in the regular case, and the case of rank-one Fano threefolds (such as projective space).  Following a brief introduction to the notion of Poisson subspaces, we discuss Bondal's conjecture on the dimensions of degeneracy loci on Poisson Fano manifolds.  We close with a discussion of log symplectic manifolds with simple normal crossings degeneracy divisor, including a new proof of the classification in the case of rank-one Fano manifolds.   
\end{abstract}

\setcounter{tocdepth}{2}
\tableofcontents

%%%%%%%%%%%%%%%%%%%%%%%%%%%%%%%%%%%%%%%%%%%%%%%%%%%%%%%%%%%%%%%%%%%%%%%%
%%%%%%%%%%%%%%%%%%%%%%%%%%%%%%%%%%%%%%%%%%%%%%%%%%%%%%%%%%%%%%%%%%%%%%%%
%%%%%%%%%%%%%%%%%%%%%%%%%%%%%%%%%%%%%%%%%%%%%%%%%%%%%%%%%%%%%%%%%%%%%%%%

\section{Introduction}

\subsection{Basic definitions and aims}
This paper is an introduction to the geometry of holomorphic Poisson structures, i.e.~Poisson brackets on the ring of holomorphic or algebraic functions on a complex manifold or algebraic variety (and sometimes on more singular objects, such as schemes and analytic spaces).  It grew out of a mini-course delivered by the author at the ``Poisson 2016'' summer school in Geneva.  The theme for the course was how the methods of algebraic geometry can be used to construct and classify Poisson brackets.  Hence this paper also serves a second purpose: it is an overview of results on the classification of projective Poisson manifolds that have been obtained by several authors over the past couple of decades, with some added context for the results and the occasional new proof.  

When one first encounters Poisson brackets, it is often in the setting of classical mechanics, which is usually formulated using $C^\infty$ manifolds.  However, there are many situations in which one naturally encounters Poisson brackets that are actually holomorphic:
\begin{itemize}
\item Classical integrable systems
\item Moduli spaces in gauge theory, algebraic geometry and low-dimensional topology
\item Noncommutative ring theory
\item Lie theory and geometric representation theory
\item Cluster algebras
\item Generalized complex geometry
\item String theory
\item \ldots
\end{itemize}
One is therefore lead to develop the holomorphic theory in parallel with its $C^\infty$ counterpart. While the two settings have much in common, there are also many important differences.  These differences are a consequence of the rigidity of holomorphic and algebraic functions, and they will play a central role in our discussion.

The first (albeit minor) difference comes already in the definition: while a Poisson bracket on a $C^\infty$ manifold is simply defined by a Poisson bracket on the ring of global smooth functions, this definition is no longer appropriate in the holomorphic setting.  The problem is that a complex manifold $\X$ may have very few global holomorphic functions; for instance, if $\X$ is compact and connected, then the maximum principle implies that every holomorphic function on $\X$ will be constant.  Thus, in order to define the bracket on $\X$, we must define it in local patches which can be glued together in a globally consistent way.  In other words, we should replace the ring of global functions with the corresponding sheaf:
\begin{definition}\label{def:poisson}
Let $\X$ be a complex manifold or algebraic variety, and denote by $\cO_\X$  its sheaf of holomorphic functions.  A \defn{(holomorphic) Poisson structure} on $\X$ is a $\CC$-bilinear operation
\[
\{\cdot,\cdot\} : \cO_\X \times \cO_\X \to \cO_\X
\]
satisfying the usual axioms for a Poisson bracket.  Namely, for all $f,g,h\in\cO_\X$, we have the following identities
\begin{enumerate}
\item Skew-symmetry:
\[
\{f,g\} = -\{g,f\}
\]
\item Leibniz rule:
\[
\{f,gh\} = \{f,g\}h+g\{f,h\}
\]
\item  Jacobi identity:
\[
\{f,\{g,h\}\} + \{g,\{h,f\}\} + \{h,\{f,g\}\}= 0.
\]
\end{enumerate}
\end{definition}

Let us denote by $\tshf{\X}$ the sheaf of holomorphic vector fields on $\X$.  These are holomorphic sections of the tangent bundle of $\X$, or equivalently, they are derivations of $\cO_\X$.  As in the $C^\infty$ setting, a holomorphic Poisson bracket can be encoded in a global holomorphic bivector field
\begin{align*}
\pi &\in \Gamma(\X,\wedge^2\tshf{\X}),
\end{align*}
using the pairing between vectors and forms: $\{f,g\} = \abrac{df\wedge dg , \pi}$.
Thus, in local holomorphic coordinates $x_1,\ldots,x_n$, we have
\[
\pi = \sum_{i<j} \pi^{ij}\cvf{x_i}\wedge\cvf{x_j}
\]
where $\pi^{ij} = \{x_i,x_j\}$ denotes the Poisson brackets of the coordinates.   The Jacobi identity for the bracket is equivalent to the condition
\[
[\pi,\pi] = 0 \in \Gamma(\X,\wedge^3\tshf{\X}).
\]
on the Schouten bracket of $\pi$.  

Every function $f \in \cO_\X$ has a Hamiltonian vector field
\[
H_f = \hook{df}\pi \in \tshf{\X}
\]
which acts as a derivation on $g \in \cO_\X$ by
\[
H_f(g) = \{f,g\}.
\]
If we start at a point $p \in \X$ and apply the flows of all possible Hamiltonian vector fields, we sweep out an even-dimensional immersed complex submanifold, called the leaf through $p$.  The bivector can then be restricted to the leaf, and inverted to obtain a holomorphic symplectic form.  Thus $\X$ has a natural foliation by holomorphic symplectic leaves.  This foliation will typically be singular, in the sense that there will be leaves of many different dimensions. 

One of the major challenges in Poisson geometry is to deal in an efficient manner with the singularities of the foliation, and this is an instance where the holomorphic setting departs significantly from the $C^\infty$ one.  Indeed,  the powerful tools of algebraic geometry give much tighter control over the local and global behaviour of holomorphic Poisson structures.

Thus, our aim is to give some introduction to how an algebraic geometer might think about Poisson brackets.  We will focus on the related problems of \defn{construction} and \defn{classification}: how do we produce holomorphic Poisson structures on  compact complex manifolds, and how do we know when we have found them all?

The paper is organized as a sort of induction on dimension.  We begin in \autoref{sec:surfaces} with a detailed discussion of Poisson surfaces, focusing in particular on the projective plane, ruled surfaces and blowups, and culminating in a statement of the full birational classification~\cite{Bartocci2005,Ingalls1998}.  In \autoref{sec:threefolds}, we discuss many types of Poisson structures on threefolds---enough to cover all of the cases that appear in the classifications of regular Poisson threefolds~\cite{Druel1999}, and Poisson Fano threefolds with cyclic Picard group~\cite{Cerveau1996,Loray2013}.  \autoref{sec:degeneracy} discusses the general notions of Poisson subspaces and degeneracy loci (where the foliation has singularities).  We highlight the intriguing phenomenon of excess dimension that is commonplace for these loci, as formulated in a conjecture of Bondal.

We close in \autoref{sec:log-symp} with an introduction to log symplectic manifolds: Poisson manifolds that have an open dense symplectic leaf, but degenerate along a reduced hypersurface.  We give many  natural examples, and   discuss in some detail the case in which the hypersurface is a simple normal crossings divisor, leading to a streamlined proof of the classification~\cite{Lima2014} for Fano manifolds with cyclic Picard group.  We also mention a recent result of the author~\cite{Pym2016} in the case of elliptic singularities.  Although the material in this section was not covered in the lecture series, the author felt that it should be included  here, given the focus on classification.

The lecture series on which this article is based was intended for an audience that already has some familiarity with Poisson geometry, but has potentially had less exposure to algebraic or complex geometry.  We have tried to keep this article similarly accessible.  Thus, we recall a number of basic algebro-geometric concepts (at the level of Griffiths and Harris' book~\cite{Griffiths1994}) by illustrating how they arise in our specific context.  On the other hand, we hope that because of the focus on examples, experts in algebraic geometry will also find this article useful, both as an introduction to the geometry of Poisson brackets, and as a guide to the growing literature on the subject.  Either way, the reader may find it helpful to complement the present article with a  more theoretical treatment of the foundations, such as Polishchuk's paper~\cite{Polishchuk1997}, the books by Dufour--Zung~\cite{Dufour2005} and Laurent-Gengoux--Pichereau--Vanhaecke~\cite{Laurent-Gengoux2013}, and the author's PhD thesis~\cite{Pym2013}.

\subsection{What is meant by classification?}

Before we begin, we should make a few remarks to clarify what we mean by ``classification'' of Poisson structures.  There are essentially two types of classifications: local and global. 

In the local case, one is looking for nice local normal forms for Poisson brackets---essentially, coordinate systems in which the Poisson bracket takes on a simple standard form.  While these issues will come up from time to time, they will not be the main focus of this paper.  We instead encourage the interested reader to consult the book~\cite{Dufour2005} for an introduction.

In the global case,  one would ideally like a list of all compact holomorphic Poisson manifolds (up to isomorphism), but there are far too many such manifolds to have any reasonable hope of classification.  One way to get some control over the situation is to look for a birational classification, i.e.~a list of Poisson manifolds from which all others can be constructed by simple transformations, such as blowing up.  As we shall see, this program has been completely realized in the case of surfaces.

Another way to get some control is to focus our attention on classifying all Poisson structures on a fixed compact complex manifold $\X$.  To see that this is potentially a tractable problem, we observe that the space $\Gamma(\X,\wedge^2\tshf{\X})
$ of bivector fields is a finite-dimensional complex vector space.  When written in a basis for this vector space, the integrability condition $[\pi,\pi] = 0$ amounts to a finite collection of homogeneous quadratic equations.  These equations therefore determine an algebraic subvariety
\[
\Pois(\X) \subset \Gamma(\X,\wedge^2\tshf{\X}).
\]
so that we can try to understand the \defn{moduli space of Poisson structures on $\X$}, i.e.~the quotient
\[
\frac{\Pois(\X)}{\Aut(\X)},
\]
where $\Aut(\X)$ is the  group of holomorphic automorphisms of $\X$.

In general, the moduli space will have many (but only finitely many) irreducible components, corresponding to qualitatively different types of Poisson structures on $\X$.   A reasonable goal for the classification is to list of all of the irreducible components in the moduli space, and describe the geometry of the Poisson structures that correspond to some dense open subset in each component.  Achieving this goal gives one a fairly good understanding of the qualitative behaviour of Poisson structures on $\X$.  As we shall see, this programme has now been carried out for several important three-dimensional manifolds, particularly the projective space $\PP^3$.  But in higher dimensions, one encounters many new challenges, and despite some recent progress for $\PP^4$, the classification remains open even in that case.

\subsection{The role of Fano manifolds}

In this paper, we will focus mostly on projective manifolds, i.e.~compact complex manifolds $\X$ that admit holomorphic embeddings in some projective space $\PP^n$, although we will also make some remarks in the non-projective setting.  In fact, a number of the main results we shall mention pertain to a particular subclass of projective manifolds (the Fano manifolds), so we should briefly explain why they are important from the perspective of Poisson geometry.

Let us recall that, roughly speaking, the minimal model program seeks to build an arbitrary projective manifold out of simpler pieces (up to birational equivalence).  The basic building blocks $\X$ come in three distinct types according to their Ricci curvatures, as measured by the first Chern class $c_1(\X) \in \cohlgy[2]{\X,\ZZ}$ of the tangent bundle:
\begin{itemize}
\item \defn{Canonically polarized varieties}, for which $c_1(\X) < 0$;
\item \defn{Calabi--Yau varieties}, for which $c_1(\X) = 0$; and
\item \defn{Fano varieties}, for which $c_1(\X) > 0$.
\end{itemize}
The notation $c_1(\X) > 0$ means that $c_1(\X)$ is an ample class, i.e.~that it can be represented by a K\"ahler form, or equivalently that $\int_\Y c_1(\X)^{\dim \Y} > 0$ for every closed subvariety $\Y \subset \X$. Similarly, $c_1(\X) < 0$ means that $-c_1(\X)$ is ample.

It therefore seems sensible to focus on the Poisson geometry of each of these three types of manifolds separately.  First of all, while canonically polarized manifolds exhibit rich and interesting algebraic geometry, they do not offer much in the way of Poisson geometry.  Indeed, the Kodaira--Nakano vanishing theorem implies that they admit no nonzero bivector fields, and hence the only Poisson bracket on such a manifold is identically zero.

Meanwhile, on a Calabi--Yau manifold $\X$, the line bundle $\det \tshf{\X}$ is holomorphically trivial.  By choosing a trivialization, we obtain an isomorphism $\wedge^\bullet\tshf{\X} \cong \forms[\dim \X - \bullet]_\X$, so that Poisson bivectors may alternatively be viewed as global holomorphic forms of degree $\dim \X - 2$.  To see what such Poisson structures can look like, we recall the following fundamental fact.
\begin{theorem}[\cite{Beauville1983,Bogomolov1974,Michelsohn1982}]
Suppose that $\X$ is a compact K\"ahler manifold with $c_1(\X) = 0$.  Then $\X$ has a finite cover that is a product
\[
\T \times \prod_j \bS_j \times \prod_k \Y_k
\]
of a compact complex torus $\T \cong \CC^n/\Lambda$, irreducible holomorphic symplectic manifolds $\bS_j$, and Calabi--Yau manifolds $\Y_k$ of dimension $d_k \ge 3$ such that $\Gamma(\Y_k,\forms[p]_{\Y_k}) = 0$ for $0 < p < d_k$.
\end{theorem}
One can easily show that any Poisson structure on $\X$ must similarly decompose as a product, so we may as well seek to understand the Poisson geometry of the individual factors.  The simplest are the factors $\Y_k$, which evidently admit no Poisson structures whatsoever.  The next simplest is the torus $\T$; its Poisson structures are all induced by constant bivectors on $\CC^n$.   Finally, we have the irreducible symplectic manifolds, which are the same thing as compact hyper-K\"ahler manifolds.  They admit a unique Poisson structure (up to rescaling), and it is induced by a holomorphic symplectic form.  The theory of these manifolds is beautiful, and quite well developed, but we shall not discuss it in this paper.  We refer, instead, to the survey by Huybrechts~\cite{Huybrechts2003}, although we note that the subject has evolved in subsequent years.  For our purposes, the upshot of this discussion is that Poisson geometry on Calabi--Yau manifolds is essentially symplectic geometry ``in families'', by which we mean that all of the symplectic leaves have the same dimension.

Thus, we are left with the Fano manifolds, typical examples of which include projective spaces $\PP^n$, hypersurfaces in $\PP^n$ of low degree, Grassmannians, flag manifolds, and various moduli spaces in gauge theory and algebraic geometry.   We refer to \cite{Iskovskikh1999} for an overview of the general structure and classification of these manifolds.

While Fano manifolds admit no global holomorphic differential forms of positive degree,  they often do carry Poisson structures, and it turns out that the symplectic foliation of a nontrivial Poisson structure on a Fano manifold is always singular.  Indeed, this foliation is typically very complicated, even for the simplest case $\X = \PP^n$.  Thus, amongst the basic building blocks listed above, it is only the Fano manifolds that truly exhibit the difference between symplectic structures and general Poisson structures.  In the past several years, there have been a number of nontrivial results on the structure and classification of Poisson Fano manifolds, but the subject is still in infancy compared with the symplectic case.  No doubt, some new conceptual understanding will be required in order to make significant progress in this area.

\paragraph{Acknowledgements:} I would like to thank the organizers of the Poisson 2016 Summer School for inviting me to give these lectures, and for encouraging me to produce this survey.  I would also like to thank Victor Mouquin and Mykola Matviichuk for acting as TAs for the mini-course, and the many participants of the school for their interest in the material and their insightful questions.  Finally, I would like to thank Georges Dloussky for his very helpful correspondence regarding class VII Poisson surfaces.  At various stages, this work was supported by EPSRC Grant EP/K033654/1; a Junior Research Fellowship at Jesus College, Oxford; and a William Gordon Seggie Brown Research Fellowship at the University of Edinburgh.  The three-dimensional renderings were created using \textsc{surfex}~\cite{Holzer2008}.

\section{Poisson surfaces}
\label{sec:surfaces}

\subsection{Basics of Poisson surfaces}

Recall that a complex surface is simply a complex manifold of complex dimension two.  In other words, every point has a neighbourhood isomorphic to an open set in $\CC^2$, and the transition maps between the different coordinate charts are holomorphic.  A \defn{(complex) Poisson surface} is a complex surface $\X$ equipped with a holomorphic Poisson bracket as in \autoref{def:poisson}.  In this section, we will examine the local and global behaviour of Poisson structures on surfaces.  In the end, we will arrive at the full birational classification:  a relatively short list of  Poisson surfaces from which all others can be obtained by simple modifications, known as blowups.

\subsubsection{Local structure}

To warm up, let us consider the local situation of a Poisson bracket defined in a neighbourhood of the origin in $\CC^2$.  Using the standard coordinates $x,y$ on $\CC^2$, we can write
\[
\{x,y\} = f(x,y)
\]
where $f$ is a holomorphic function.  The corresponding bivector field is given by
\[
\pi = f \cvf{x}\wedge \cvf{y}.
\]
Thus, once we have fixed our coordinates, the Poisson bracket is determined by the single function $f$.  Because of the dimension, this bivector automatically satisfies $[\pi,\pi]=0$, so the Jacobi identity does not impose any constraints on the function $f$.

Away from the locus where $f$ vanishes, we can invert $\pi$ to obtain a symplectic two-form
\[
\omega = \frac{dx\wedge dy}{f}.
\]
Applying the holomorphic version of Darboux's theorem, we may find local holomorphic coordinates $p$ and $q$ in which 
\[
\omega = dp \wedge dq,
\]
or equivalently
\[
\pi = \cvf{q}\wedge\cvf{p}.
\]
Thus, the local structure of $\pi$ is completely understood in this case.

But things are more complicated near the zeros of  $f$.  Without loss of generality, let us assume that $f$ vanishes at the origin. Recall that the zero locus of a single holomorphic function always has complex codimension one.  Hence if $f(0,0)=0$, there must actually be a whole complex curve $\D \subset \CC^2$, passing through the origin, on which $f$ vanishes.  Let us suppose for simplicity that $f$ is a polynomial.  Then it will have  a factorization into a finite number of  irreducible factors:
\[
f = f_1^{k_1} f_2^{k_2} \cdots f_n^{k_n}.
\]
Therefore $\D$ will be a union of the vanishing sets $\D_1,\ldots,\D_n$ of $f_1,\ldots,f_n$, the so-called \defn{irreducible components}.  (If $f$ is not polynomial, then $\D$ may have an infinite number of irreducible components, but the number of components will be locally finite, i.e. there will only be finitely many in a given compact subset of $\CC^2$.)  
Several examples of Poisson structures and the corresponding curves are shown in \autoref{fig:planar}; these pictures represent real two-dimensional slices of the four-dimensional space $\CC^2$.

\begin{figure}[h]
\centering
\begin{subfigure}{0.4\textwidth}
\centering
\begin{tikzpicture}[scale=1.5]
\draw[thin,fill=lightgrey] (-1,-1) rectangle (1,1);
\draw[very thick,color=red] (0,-1) -- (0,1);
\end{tikzpicture}
\caption{$x\,\cvf{x}\wedge\cvf{y}$}
\end{subfigure}
\begin{subfigure}{0.4\textwidth}
\centering
\begin{tikzpicture}[scale=1.5]
\draw[thin,fill=lightgrey] (-1,-1) rectangle (1,1);
\draw[very thick,color=red] (-1,0) -- (1,0);
\draw[very thick,color=red] (0,-1) -- (0,1);
\end{tikzpicture}
\caption{$xy\,\cvf{x}\wedge\cvf{y}$}
\end{subfigure}

\vspace{0.5cm}
\begin{subfigure}{0.4\textwidth}
\centering
\begin{tikzpicture}[scale=1.5]
\draw[thin,fill=lightgrey] (-1,-1) rectangle (1,1);
\draw[domain=-1:1,smooth,very thick,color=red] plot ({\x*\x},{\x*\x*\x});
\end{tikzpicture}
\caption{$(x^3-y^2)\,\cvf{x}\wedge\cvf{y}$}
\end{subfigure}
\begin{subfigure}{0.4\textwidth}
\centering
\begin{tikzpicture}[scale=1.5]
\draw[thin,fill=lightgrey] (-1,-1) rectangle (1,1);
\draw[very thick,color=red] (0,-1) -- (0,1);
\draw[very thick,dashed,color=red] (0.05,-1) -- (0.05,1);
\draw[domain=-1:1,smooth,very thick,color=red] plot ({\x},{\x*\x});
\end{tikzpicture}
\caption{$x^2(y-x^2)\,\cvf{x}\wedge\cvf{y}$}
\end{subfigure}
\caption{Zero loci of Poisson structures on $\CC^2$; the dashed line indicates the presence of a component with multiplicity two.}\label{fig:planar}
\end{figure}
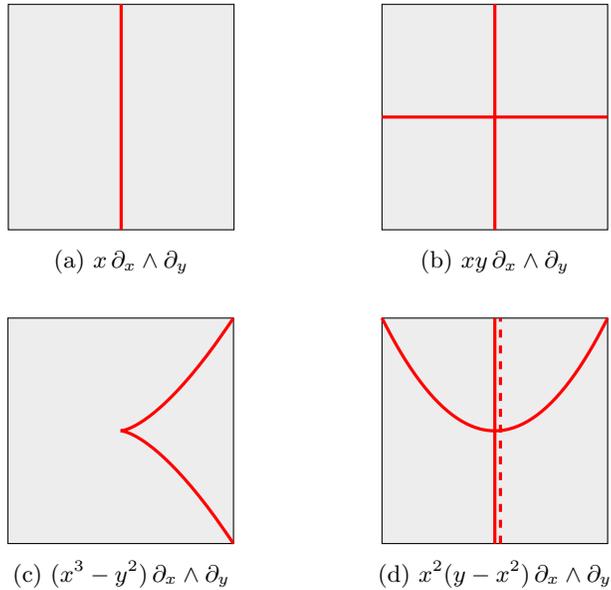

If two bivectors $\pi$ and $\pi'$ have the same zero set, counted with multiplicities, then they must differ by an overall factor
\[
\pi  = g \pi',
\]
where $g$ is a nonvanishing holomorphic function.  While the behaviour of these two bivectors is clearly very similar, they will not, in general, be isomorphic, i.e.~we can not take one to the other by a suitable coordinate change on $\CC^2$.  Here is a simple example:

\begin{exercise}\label{exer:noniso}
Given any constant $\lambda \in \CC$, define a Poisson bracket on $\CC^2$ by the formula
\begin{align}
\{x,y\}_\lambda = \lambda \,xy, \label{eqn:nodal}
\end{align}
where $x,y$ are the standard coordinates on $\CC^2$.  Show that the brackets $\{\cdot,\cdot\}_{\lambda}$ and $\{\cdot,\cdot\}_{\lambda'}$ are isomorphic if and only if $\lambda = \pm \lambda'$.  Conclude that the isomorphism class of a Poisson structure on $\CC^2$ depends on more information than just the divisor on which it vanishes. \qed
\end{exercise}

In fact, the constant $\lambda$ appearing in \eqref{eqn:nodal} is the only addition piece of information required to understand the local structure of the bracket in this case.  More precisely, suppose that $\pi$ is a Poisson structure on a surface $\X$, and that $\D$ is the curve on which it vanishes.  Suppose that $p \in \D$ is a nodal singular point of $\D$.  (This means that, in a neighbourhood of $p$, the curve $\D$ consists of two smooth components with multiplicity one that intersect transversally at $p$.)  Then one can find a constant $\lambda \in \CC$, and coordinates $x,y$ centred at $p$ in which the Poisson bracket has the form \eqref{eqn:nodal}.

In general, finding a local normal form for the bracket in a neighbourhood of a singular point of $\D$ can be complicated.  The main result in this direction is a theorem of Arnold, which gives a local normal form in the neighbourhood of a simple singularity of the curve $\D$.  Since we shall not need the precise form of the result, we shall omit it.  We refer to the original article~\cite{Arnold1987} for details; see also \cite[Section 2.5.1]{Dufour2005} and \cite[Section 9.1]{Laurent-Gengoux2013}.

\subsection{Poisson structures on the projective plane}\label{sec:p2}

We will be mainly concerned with Poisson structures on compact  complex surfaces.  The most basic example is the projective plane:
\[
\PP^2 = \{\textrm{lines through 0 in }\CC^3\} = (\CC^3 \setminus \{0\}) / \CC^*,
\]
where $\CC^*$ acts by rescaling.  The equivalence class of a point $(x,y,z) \in \CC^3\setminus \{0\}$ is denoted by $[x:y:z] \in \PP^2$, 
so that
\[
[x:y:z] = [\lambda x:\lambda y:\lambda z]
\]
for all $\lambda \in \CC^*$. 

Let us examine the possible behaviour of a Poisson bracket on $\PP^2$, defined by the global holomorphic bivector
\[
\pi \in \Gamma(\PP^2,\wedge^2\tshf{\PP^2}).
\]
We will determine the behaviour of $\pi$ in the three standard coordinate charts on $\PP^2$, given by the open dense sets
\begin{align*}
\U_1 &=  \set{ [x:y:z]  \in \PP^2}{x \ne 0} \\
\U_2 &= \set{ [x:y:z]  \in \PP^2}{y \ne 0} \\
\U_3 &= \set{ [x:y:z]  \in \PP^2}{z \ne 0},
\end{align*}
For example, the coordinates associated to $\U_1$ are given by
\begin{align*}
u([x:y:z]) &= \frac{y}{x} & v([x:y:z]) &= \frac{z}{x},
\end{align*}
giving an isomorphism $\U_1 \cong \CC^2$.  Thus, any Poisson bracket on $\PP^2$ must be represented in these coordinates by
\[
\{u,v\} = f
\]
for some holomorphic function $f$ defined on all of $\CC^2$.  In other words, the bivector has the form
\[
\pi|_{\U_1} = f \cvf{u}\wedge \cvf{v}.
\]

Meanwhile,  in the chart $\U_2$ we have the coordinates   $\widetilde u, \widetilde v$ given by
\begin{align*}
\tu([x:y:z]) &= \frac{x}{y} & \tv([x:y:z]) &= \frac{z}{y}.
\end{align*}
They are related to the original coordinates by $\tu = u^{-1}$ and $\tv = u^{-1}v$ on the overlap of the charts.  We can therefore compute
\begin{align*}
\{\tu,\tv \} &= \{u^{-1},u^{-1}v\} \\
&= \{u^{-1},u^{-1}\}v + u^{-1}\{u^{-1},v\} \\
&= 0 + u^{-1}\cdot \rbrac{-u^{-2}\{u,v\}} \\
&= -u^{-3} f(u,v) \\
&= -\tu^3 f(\tu^{-1},\tu^{-1}\tv).
\end{align*}
Thus, taking the Taylor expansion of $f$, we find
\[
\pi|_{\U_2} = \rbrac{\tu^3 \sum_{j,k=0}^\infty a_{jk}\tu^{-(j+k)}\tv^k}\cvf{\tu}\wedge \cvf{\tv}
\]
for some coefficients $a_{jk}\in \CC$.  Since $\pi$ is holomorphic on the whole chart $\U_2$ we must have that $a_{jk} = 0$ whenever $j+k > 3$; otherwise, $\pi$ would have a pole when $\tu = 0$.  These constraints are equivalent to requiring that $f$ be a polynomial in $u$ and $v$ of degree at most three.

A similar calculation in the other chart evidently yields the same result.  We conclude that a Poisson bracket on $\PP^2$ is described in any affine coordinate chart by a cubic polynomial.  Conversely, given a cubic polynomial in an affine chart, we obtain a Poisson structure on $\PP^2$.

Once again, the zero set $\D = \Zeros(\pi)$ determines $\pi$ up to rescaling by a global nonvanishing holomorphic function, but now every such function is constant because $\PP^2$ is compact.  Converting from the affine coordinates $(u,v)$ to the ``homogeneous coordinates'' $x,y,z$, we arrive at the following
\begin{proposition}\label{prop:p2cubic}
A Poisson structure $\pi$ on $\PP^2$ always vanishes on a cubic curve, meaning that
\[
\D= \set{[x:y:z] \in \PP^2}{F(x,y,z) = 0}
\]
where $F$ is a homogeneous polynomial of degree three.  This curve, and  the multiplicities of its components, determine $\pi$ up to rescaling by a nonzero constant.
\end{proposition}

If $\pi$ is suitably generic, the curve $\D$ will be smooth.  We recall that, in this case, it must be an elliptic curve---a Riemann surface that is topologically a torus.  (More correctly, it is a smooth curve of genus one; typically one says that an elliptic curve is a genus one curve with a chosen base point, but we shall be loose about the distinction.)

More degenerate scenarios are possible, in which the curve becomes singular.  The full classification of all possible cubic curves is classical: up to projective equivalence, there are only nine possible behaviours, as illustrated in \autoref{fig:cubic}.

\begin{figure}[h]\centering
\begin{subfigure}{0.3\textwidth}\centering
\begin{tikzpicture}[scale=\cubicscale]
\def\mint{-3.5}%-1.769292354}
\def\maxt{3.5}
\def\a{1}
\def\b{2}
\clip[draw] (0,0) circle (4);
\draw[fill=lightgrey] (0,0) circle (4);
\def\re{(108*\x*\x-108*\b+12*sqrt(12*\a*\a*\a+81*\x*\x*\x*\x-162*\x*\x*\b+81*\b*\b))^(1/3)}

\def\x{\t}
\def\y{ (\re/6 - 2*\a/\re) } %(\re^2-12*\a)/\re/6}}
\def\tscale{0.1}
\def\tx{\tscale*(3*\y^2+\a)}
\def\ty{\tscale*(2*\x)}
\draw[ultra thick,red,domain=\mint:\maxt,smooth,variable=\t] plot ({\y},{\x});
%\foreach \t in {-3,-2,-1.5,-1,0,1,1.5,2,3}
%\draw[color=sxred,->,ultra thick] ({\x},{\y}) -- ({\x+\tx},{\y+\ty});
\end{tikzpicture}
\caption{smooth (elliptic)}
\end{subfigure}
\begin{subfigure}{0.3\textwidth}\centering
\begin{tikzpicture}[scale=\cubicscale]
\clip[draw] (0,0) circle (4);
\draw[fill=lightgrey] (0,0) circle (4);
\draw[ultra thick,red,domain=-2:2,smooth,variable=\t] plot ({2*\t*\t-2},{2*\t*\t*\t-2*\t});
\end{tikzpicture}
\caption{node}
\end{subfigure}
\begin{subfigure}{0.3\textwidth}\centering
\begin{tikzpicture}[scale=\cubicscale]
\clip[draw] (0,0) circle (4);
\draw[fill=lightgrey] (0,0) circle (4);
\draw[ultra thick,red,domain=-2:2,smooth,variable=\t] plot ({\t*\t},{\t*\t*\t/2});
\end{tikzpicture}
\caption{cusp}
\end{subfigure}

\vspace{0.5cm}
\begin{subfigure}{0.3\textwidth}\centering
\begin{tikzpicture}[scale=\cubicscale]
\clip[draw] (0,0) circle (4);
\draw[fill=lightgrey] (0,0) circle (4);
\draw[ultra thick,color=red] (0,0) circle (2.5);
\draw[ultra thick,color=red] (-4,0) -- (4,0);
\end{tikzpicture}
\caption{conic and line}
\end{subfigure}
\begin{subfigure}{0.3\textwidth}\centering
\begin{tikzpicture}[scale=\cubicscale]
\clip[draw] (0,0) circle (4);
\draw[fill=lightgrey] (0,0) circle (4);
\draw[ultra thick,color=red] (0,0) circle (2.5);
\draw[ultra thick,color=red] (-4,2.5) -- (4,2.5);
\end{tikzpicture}
\caption{conic and tangent line}
\end{subfigure}
\begin{subfigure}{0.3\textwidth}\centering
\begin{tikzpicture}[scale=\cubicscale]
\clip[draw] (0,0) circle (4);
\draw[fill=lightgrey] (0,0) circle (4);
\draw[ultra thick,color=red] (-5,-1) -- (5,-1);
\draw[ultra thick,color=red] (3.4641,5) -- (-3.4641,-3);
\draw[ultra thick,color=red] (-3.4641,5) -- (3.4641,-3);
\end{tikzpicture}
\caption{triangle}
\end{subfigure}

\vspace{0.5cm}
\begin{subfigure}{0.3\textwidth}\centering
\begin{tikzpicture}[scale=\cubicscale]
\clip[draw] (0,0) circle (4);
\draw[fill=lightgrey] (0,0) circle (4);
\draw[ultra thick,color=red] (-5,0) -- (5,0);
\draw[ultra thick,color=red] (60:5) -- (240:5);
\draw[ultra thick,color=red] (120:5) -- (-60:5);
\end{tikzpicture}
\caption{lines through a point}
\end{subfigure}
\begin{subfigure}{0.3\textwidth}\centering
\begin{tikzpicture}[scale=\cubicscale]
\clip[draw] (0,0) circle (4);
\draw[fill=lightgrey] (0,0) circle (4);
\draw[ultra thick,color=red] (-5,0) -- (5,0);
\draw[dashed,thick, color=red] (-5,0.3) -- (5,0.3);
\draw[ultra thick,color=red] (0,-5) -- (0,5);
\end{tikzpicture}
\caption{line and a double line}
\end{subfigure}
\begin{subfigure}{0.3\textwidth}\centering
\begin{tikzpicture}[scale=\cubicscale]
\clip[draw] (0,0) circle (4);
\draw[fill=lightgrey] (0,0) circle (4);
\draw[ultra thick,color=red] (-5,0) -- (5,0);
\draw[dashed,thick, color=red] (-5,0.3) -- (5,0.3);
\draw[dashed,thick, color=red] (-5,0.6) -- (5,0.6);
\end{tikzpicture}
\caption{triple line}
\end{subfigure}

\caption{The various types of cubic curves in $\PP^2$.  Each curve determines a Poisson structure on $\PP^2$ up to rescaling by an overall constant.}
\label{fig:cubic}
\end{figure}
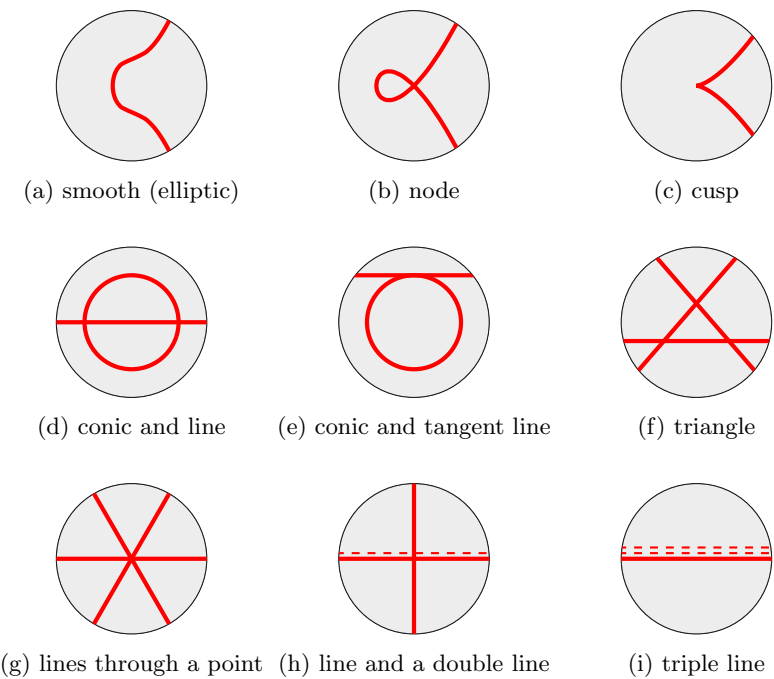
%
%\paragraph{Abelian surfaces}\brent{CONTINUE EDITS FROM HERE}
%
%
%
%There are many other compact Poisson surfaces.  For example, consider the vector field on the projective line $\PP^1$ defined by
%\[
%Z = x\cvf{x} 
%\]
%in the affine coordinate $[x:1]$.  Switching to the other coordinate chart $w=x^{-1}$, we find
%\[
%Z = -w\cvf{w}
%\]
%so that this vector field has a second zero at $\infty \in \PP^1$.  This vector field has a second zero at $\infty$, of order one.  Using $Z$, we may define a Poisson structure on the product
%\[
%\X = \PP^1 \times \PP^1
%\]
%by taking the product 
%\[
%\pi  = Z_1 \wedge Z_2
%\]
%where $Z_1$ and $Z_2$ are the vector fields on $\X$ constructed by putting $Z$ on each of the $\PP^1$-factors.  The zero locus is given by
%\[
%\D = \{0,\infty\} \times \PP^1 + \PP^1 \times \{0,\infty\},
%\]
%which is a rectangle of projective lines.  More generally, the zero loci of Poisson structures on $\PP^1\times \PP^1$ are the curves $\D \subset \X$ of bidegree $(2,2)$: for $p$ sufficiently generic, such a curve intersects the lines $\PP^1\times\{p\}$ and $\{p\}\times\PP^1$ in two points, counted with multiplicity.  When $\D$ is smooth, it is an elliptic curve, the same as for $\PP^2$.
%
%For another example, suppose that $\Y$ is a smooth curve and consider the cotangent bundle $\T^*\Y$.  This line bundle has a natural compactification
%\[
%\X = \overline{\T^*\Y},
%\]
%obtained by adding a section at infinity to get a $\PP^1$-bundle.  The symplectic form on $\T^*\Y$ then extends to a Poisson structure on $\X$ that vanishes to order two on the section at infinity.

\subsection{Anticanonical divisors and adjunction}

Before we continue our discussion of Poisson surfaces, it will be useful to recall some standard algebro-geometric terminology and conventions concerning divisors.  We shall be brief, so we refer the reader to \cite[Chapter 1]{Griffiths1994} for a comprehensive treatment.

\subsubsection{Divisors}

If $\X$ is a complex manifold, a \defn{divisor on $\X$} is a formal $\ZZ$-linear combination
\[
\D = \sum_{i } k_i\D_i 
\]
of irreducible hypersurfaces $\D_i \subset \X$. It is assumed that this combination is locally finite, meaning that every point $p\in\X$ has a neighbourhood $\U$ such there are only finitely indices $i$ for which $\D_i \cap \U \ne \varnothing$ and $k_i \ne 0$.   In our examples, $\D$ will simply be a finite sum.

A divisor is \defn{effective} if each coefficient $k_i$ is nonnegative.  A typical example of an effective divisor is the zero locus of a holomorphic function with coefficients given by the multiplicities of vanishing.  More globally, the zero locus of a holomorphic section of a holomorphic line bundle defines an effective divisor, and in fact all effective divisors arise in this way.

An effective divisor $\D$ is \defn{smooth at $p \in \X$} if, near $p$, it can be defined as the zero locus of a single function $f$ whose derivative is nonzero:
\[
df|_p \ne 0 \in \T^*_p\X.
\]
In this case, the implicit function theorem implies that there is an open neighbourhood $\U$ of $p$ such that $\D \cap \U \subset \U$ is a connected complex submanifold of codimension one.  Moreover the function $f$ vanishes to order one on this submanifold.

Note the convention here: an effective divisor is never smooth if it has any irreducible components that are taken with multiplicity greater than one, even if the  underlying set of points is a submanifold.  This ensures, for example, that smoothness of a divisor is preserved by small deformations of its local defining equations.  In contrast, the equation
\[
x^2=0,
\]
defining a straight line with multiplicity two in coordinates $x,y$, can evidently be deformed to the equation
\[
x(x+\epsilon y)= 0,
\]
for $\epsilon \in \CC$.  For $\epsilon \ne 0$, this deformed divisor has a singular point where the two lines $x=0$ and $x+\epsilon y = 0$ meet.

\subsubsection{Anticanonical divisors}

 Recall that on any complex manifold, whatever the dimension, the \defn{canonical line bundle} is the top exterior power of the cotangent bundle:
\[
\can_\X  = \det \tshf{\X}^\vee = \wedge^{\dim \X}\tshf{\X}^\vee
\]
So the sections of $\can_\X$ are holomorphic differential forms of top degree.  The \defn{anticanonical bundle} $\acan_\X = \det \tshf{\X}$ is the dual of the canonical bundle.
\begin{definition}
A divisor  $\D$ on $\X$ that may be obtained as the zero locus of a section of $\acan_\X$ is called an \defn{(effective) anticanonical divisor}.
\end{definition}

In the particular case when $\X$ is a surface, we have that
\[
\acan_\X = \wedge^2\tshf{\X}
\]
so that a Poisson structure $\pi$ on $\X$ is simply a section of the anticanonical bundle.  In this way, we see that a Poisson structure on a compact surface $\X$ is determined up to constant rescaling by an anticanonical divisor on $\X$.  We can then rephrase \autoref{prop:p2cubic} as follows: an effective divisor on $\PP^2$ is an anticanonical divisor if and only if it is a cubic curve.  

\subsubsection{Adjunction on Poisson surfaces}
\label{sec:adjunction}

We saw in \autoref{sec:p2} that a smooth anticanonical divisor in $\PP^2$ is always an elliptic curve.  We shall now explain a geometric reason why this must be the case.  It is a special case of a general result, known as the \defn{adjunction formula}, which relates the canonical bundle of a hypersurface to the canonical bundle of the ambient manifold (see, e.g.~\cite[p.~146--147]{Griffiths1994}).

Let $(\X,\pi)$ be a Poisson surfaces, and let $\D = \Zeros(\pi)$ be the corresponding anticanonical divisor.  Suppose that $p$ is a point of $\D$.  Then $\pi$ has a well-defined derivative at $p$, giving an element
\[
d_p\pi \in \T^\vee_p\X \otimes \wedge^2 \T_p\X
\]
where $\T_p\X$ is the tangent space of $\X$ at $p$ and $\T^\vee_p\X$ is the cotangent space.  (This tensor is the one-jet of $\pi$ at $p$.)

Notice that there is a natural contraction map
\[
\T^\vee_p\X \otimes \wedge^2 \T_p\X \to \T_p\X,
\]
given by the interior product of covectors and bivectors.    Applying this contraction to $d_p\pi$ we obtain an element $Z_p \in \T_p\X$, and allowing $p$ to vary, we obtain a section
\[
Z \in \Gamma(\D,\tshf{\X}|_\D)
\]
If $x$ and $y$ are local coordinates on $\X$, then $\pi = f \cvf{x}\wedge\cvf{y}$ for a function $f$.  We then compute
\[
d\pi = df \otimes \cvf{x}\wedge\cvf{y}|_\D,
\]
so that
\[
Z = \rbrac{(\cvf{x}f)\cvf{y} - (\cvf{y}f)\cvf{x}}|_\D.
\]
From this expression, we immediately obtain the following
\begin{proposition}
The vector $Z_p$ is nonzero if and only if $\D$ is smooth at $p$.  In this case, it is tangent to $\D$, giving a basis for the one-dimensional subspace $\T_p\D \subset \T_p\X$.
\end{proposition}

\begin{proof}
The point $p\in\D$ is smooth if and only if $df|_p$ is nonzero, where $f$ is a local defining equation for $\D$.  It is evident from the local calculation above that this is equivalent to the condition $Z_p \ne 0$.  It is also easy to see that $Z_p(f) = 0$, but this is exactly the condition for $Z_p$ to be tangent to $\D$.
\end{proof}

Using the fact that elliptic curves are the only compact complex curves whose tangent bundles are trivial, we obtain the following consequence.
\begin{corollary}
If $\D$ is smooth, then the derivative of the Poisson structure induces a nonvanishing vector field
\[
Z \in \Gamma(\D,\tshf{\D})
\]
In particular, if $\D$ is smooth and compact, then every connected component of $\D$ is an elliptic curve.
\end{corollary}

\begin{definition}
The vector field $Z \in \Gamma(\D,\tshf{\D})$ is called the \defn{modular residue} of the Poisson structure $\pi$.
\end{definition}

\begin{remark}
When $\D$ is singular, we can still make sense of vector fields on $\D$ as  derivations of functions (see  \autoref{sec:subspaces}).  In this way one can make sense of the modular residue even at the singular points. 
\end{remark}

\subsection{Poisson structures on ruled surfaces}
\label{sec:ruled}

We now describe some more examples of compact complex surfaces.  We recall that a surface $\X$ is \defn{(geometrically) ruled} if it is the total space of a $\PP^1$-bundle over a smooth compact curve $\Y$.  This is equivalent to saying that $\X = \PP(\cE)$ is the projectivization of a rank-two holomorphic vector bundle $\cE$ on $\Y$.  Not every ruled surface carries a Poisson structure, but there are several that do.  In this section, we will describe their classification.

 As is standard in algebraic geometry, we make no notational distinction between a holomorphic vector bundle and its locally free sheaf of holomorphic sections.  Thus, for example,  $\cO_\Y$ denotes both the sheaf of holomorphic functions, and the trivial line bundle $\Y \times \CC \to \Y$.

\subsubsection{Compactified cotangent bundles}\label{sec:cotan-compact}

The cotangent bundle of any smooth curve is a symplectic surface.  It can be compactified to obtain a Poisson ruled surface in several ways, which we now describe.

\paragraph{Simplest version:} To begin, suppose that $\Y$ is a smooth compact curve.  Then we may compactify the cotangent bundle by adding a point at infinity in each fibre.  More precisely, we consider the rank-two bundle $\cE = \cotshf{\Y} \oplus \cO_\Y$. Its projectivization $\X = \PP(\cotshf{\Y} \oplus \cO_\Y)$ has an open dense subset isomorphic to $\cotshf{\Y}$ given by the embedding
\[
\mapdef{i}{\cotshf{\Y}}{\PP(\cotshf{\Y} \oplus \cO_\Y) }
{\alpha}{\mathrm{span}\,(\alpha,1).}
\]
Its complement is the locus
\[
\bS = \{[*:0] \in \PP(\cotshf{\Y} \oplus \cO_\Y)\},
\]
obtained by taking the point at infinity in each $\PP^1$-fibre.  Thus $\bS$  projects isomorphically to $\Y$, giving a section of the $\PP^1$-bundle---the ``section at infinity''.

Using local coordinates, it is easy to see that the Poisson structure on $\cotshf{\Y}$ extends to all of $\X$.  Indeed, if $y$ is a local coordinate on $\Y$ and $x = \cvf{y}$ is the corresponding coordinate on the fibres of $\cotshf{\Y}$, then the bivector has the standard Darboux form $\pi = \cvf{x}\wedge \cvf{y}$.  Switching to the coordinate $z = x^{-1}$ at infinity in the $\PP^1$ fibres, we find
\[
\pi = -z^{2}\cvf{z}\wedge\cvf{y}.
\] 
Thus $\pi$ gives a Poisson structure on $\X$ whose anticanonical divisor is given by
\[
\Zeros(\pi) = 2 \bS \subset \X,
\]
the section at infinity taken with multiplicity two.

\paragraph{Twisting by a divisor:} We can modify the previous construction by introducing a nontrivial effective divisor $\D$ on the compact curve $\Y$, i.e.~a collection of points in $\Y$ with positive multiplicities.  This modification yields a compactification of the cotangent bundle of the punctured curve $\U = \Y \setminus \D$, as follows.

Denote by $\cotshf{\Y}(\D)$ the line bundle on $\Y$ whose sections are the meromorphic one-forms with poles bounded by $\D$.  More precisely, on the punctured curve $\U$, we have
\[
\cotshf{\Y}(\D)|_{\U} \cong \cotshf{\U},
\]
but if  $p \in \D$ is a point of multiplicity $k \in \ZZ_{>0}$, and $z$ is a coordinate centred at $p$, then we have a basis
\[
\cotshf{\Y}(\D) \cong \abrac{ \frac{dz}{z^{k}} }
\]
in a neighbourhood of $p$.  The symplectic structure on $\cotshf{\Y\setminus \D}$ then extends to a Poisson structure $\pi$ on the ruled surface
\[
\X = \PP(\cotshf{\Y}(\D) \oplus \cO_\Y)
\]
that vanishes both at infinity, and on the fibres over $\D$.  More precisely, the anticanonical divisor has the form
\[
\Zeros(\pi) = \sum_{p \in \D} k_p\F_p + 2\bS
\]
where $\F_p  \subset \X$ is the fibre over $p \in \Y$ as shown in \autoref{fig:cotan-compact}.

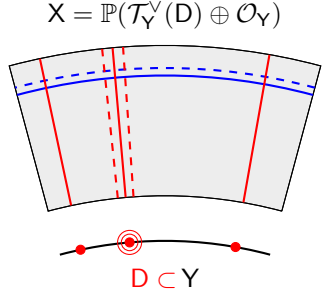
\begin{figure}[h]
\begin{center}
\begin{tikzpicture}[scale=2]
\def\t{15}
\def\r{3}
\def\R{4}
\def\s{3.8}
\def\c{2.7}
\def\PA{80}
\def\PB{95}
\def\PC{102}
\draw[thick,black] (90-\t:\c) arc (90-\t:90+\t:\c);
\draw (90:\c-0.25) node {$\textcolor{red}{\D\subset}\,\Y$};
\draw (90:\R+0.2) node {$\X = \PP(\cotshf{\Y}(\D)\oplus\cO_\Y)$};
\draw[red,fill] (\PA:\c) circle (0.03);
\draw[red,fill] (\PB:\c) circle (0.03);
\draw[red] (\PB:\c) circle (0.055);
\draw[red] (\PB:\c) circle (0.08);
\draw[red,fill] (\PC:\c) circle (0.03);
\draw[fill=lightgrey] (90-\t:\r) arc (90-\t:90+\t:\r) -- (90+\t:\R) arc (90+\t:90-\t:\R) -- (90-\t:\r);
\draw[clip] (90-\t:\r) arc (90-\t:90+\t:\r) -- (90+\t:\R) arc (90+\t:90-\t:\R) -- (90-\t:\r);
\draw[thick,blue] (90-\t:\s) arc (90-\t:90+\t:\s);
\draw[thick,dashed,blue] (90-\t:{\s+0.05}) arc (90-\t:90+\t:{\s+0.05});
\draw[thick,red] (\PA:\r) -- (\PA:\R);
\draw[thick,red] (\PB:\r) -- (\PB:\R);
\draw[thick,dashed,red] (\PB-1:\r) -- (\PB-1:\R);
\draw[thick,dashed,red] (\PB+1:\r) -- (\PB+1:\R);
\draw[thick,red] (\PC:\r) -- (\PC:\R);
\end{tikzpicture}
\end{center}
\caption{The Poisson structure on a compactified cotangent bundle vanishes to order two on the section at infinity (shown in blue) and to prescribed multiplicity on the fibres over a divisor in the base curve (shown in red).}\label{fig:cotan-compact}
\end{figure}

\begin{exercise}
Verify this description of the anticanonical divisor.  Explain what goes wrong if we try to replace $\cotshf{\Y}(\D)$  with the bundle $\cotshf{\Y}(-\D)$ of forms that vanish on $\D$ instead of having poles. \qed
\end{exercise}

\paragraph{Jet bundles:}\label{sec:jet-p1}
Another variant is obtained by working with extensions of line bundles, instead of direct sums.  Suppose that $\cE$ is a rank-two bundle that sits in an exact sequence
\[
\xymatrix{
0 \ar[r] & \cotshf{\Y} \ar[r]& \cE  \ar[r] &  \cO_\Y \ar[r] & 0
}
\] 
but does \emph{not} split holomorphically as a direct sum: $\cE \neq \cotshf{\Y} \oplus \cO_\Y$.  In fact, there is only one such bundle $\cE$ on $\Y$, up to isomorphism; it can be realized explicitly as
\[
\cE = \cJ^1\sL \otimes \sL^{\vee}
\]
where $\sL$ is any line bundle of nonzero degree and $\cJ^1\sL$ is its bundle of one-jets.  (In other words, $\cE$ is the dual of the Atiyah algebroid of $\sL$.)  The uniqueness follows from the description of extensions of vector bundles in terms of Dolbeault cohomology (see \cite[Section 5.4]{Griffiths1994}).  In this case, extensions of $\cO_\Y$ by $\cotshf{\Y}$ are parametrized by $\cohlgy[1]{\Y,\cotshf{\Y}}  \cong \cohlgy[2]{\Y,\CC}$, which is one-dimensional.

Once again, the projectivization $\X = \PP(\cE)$ has $\cotshf{\Y}$ as an open dense set, corresponding to the vectors in $\cE$ that project to $1 \in \cO_\Y$, and we obtain a Poisson structure on $\X$ that vanishes to order two on the section at infinity.

One could try to generalize this construction by introducing a nonempty divisor $\D$ on $\Y$ and considering extensions
\[
\xymatrix{
0\ar[r] & \cotshf{\Y}(\D) \ar[r] & \cE \ar[r] & \cO_\Y \ar[r] & 0.
}
\]
However, such extensions are parametrized by the vector space $\cohlgy[1]{\Y,\cotshf{\Y}(\D)}$, which by Serre duality is dual to $\cohlgy[0]{\Y,\cO_\Y(-\D)}$, the space of global holomorphic functions on $\Y$ that vanish on $\D$.  Since every global holomorphic function on $\Y$ is constant, this vector space is trivial.  We conclude that the extension is also trivial, i.e.~we have a splitting $\cE \cong \cotshf{\Y}(\D) \oplus \cO_\Y$ as previously considered, so the construction does not produce anything new.

\subsubsection{Relationship with co-Higgs fields}
\label{sec:coHiggs}
Let us now consider a general ruled surface $\X = \PP(\cE)$ over $\Y$. Let $\rho : \X \to \Y$ be the bundle projection.  There is an exact sequence
\[
\xymatrix{
0 \ar[r] & \tshf{\X/\Y} \ar[r] & \tshf{\X} \ar[r]& \rho^*\tshf{\Y} \ar[r] & 0,
}
\]
and taking determinants we find
\begin{align}
\acan_\X \cong \wedge^2\tshf{\X} \cong \rho^*\tshf{\Y} \otimes \tshf{\X/\Y}. \label{eqn:ruleddet}
\end{align}
In particular, if $f \in \cO_\Y$, then the Hamiltonian vector field of $\rho^*f$ is tangent to the fibres of $\rho$.  Thus, for each $p \in \X$, we obtain a linear map
\[
\T^\vee_p\Y \to \Gamma(\F_p,\tshf{\F_p})
\]
sending a covector at $p$ to the corresponding vector field on the fibre $\F_p$. Evidently, this linear map completely determines $\pi$ along $\F_p$.

Now since $\F_p \cong \PP(\cE_p)$ is a copy of $\PP^1$, the space $\Gamma(\F_p,\tshf{\F_p})$ of vector fields on the fibre is three-dimensional.  (In an affine coordinate $z$ on the fibre, the vector fields $\cvf{z},z\cvf{z}$ and $z^2\cvf{z}$ give a basis.)  Hence as $p$ varies, these spaces assemble into a rank-three vector bundle over $\Y$---the so-called \defn{direct image} $\rho_*\tshf{\X/\Y}$.  In this way, we see that a Poisson structure on the surface $\X$ is equivalent to a vector map bundle map $\cotshf{\Y} \to \rho_*\tshf{\X/\Y}$ on the curve $\Y$.

This is a bit abstract, but fortunately the bundle $\rho_*\tshf{\X/\Y}$ has a more concrete description.  Indeed, if we think of endomorphisms of $\cE$ as infinitesimal symmetries, then we get an identification
\[
\rho_*\tshf{\X/\Y} \cong \cEnd_0(\cE)
\]
where $\cEnd_0(\cE)$ is the bundle of traceless endomorphisms of $\cE$.  The zeros of a vector field are identified with the points in $\PP(\cE)$ determined by the eigenspaces of the corresponding endomorphism.  

We therefore arrive at the following
\begin{theorem}\label{thm:p1bundle}
Let $\X = \PP(\cE)$ be the projectivization of a rank-two bundle $\cE$ over a smooth curve $\Y$.  Then we have a canonical isomorphism
\[
\Gamma(\X,\wedge^2\tshf{\X}) \cong \Gamma(\Y,\tshf{\Y}\otimes\cEnd_0(\cE)),
\]
so that Poisson structures on $\X$ are in canonical bijection with bundle maps $\cotshf{\Y} \to \cEnd_0(\cE)$ on $\Y$.
\end{theorem}

\begin{example}\label{ex:split-ell}
Suppose that $\Y$ is a smooth curve of genus one, i.e.~an elliptic curve, and let $Z \in \Gamma(\tshf{\Y})$ be a nonzero vector field on $\Y$.  Let $\sL$ be a line bundle on $\Y$, and consider the rank-two bundle $\cE = \cO_\Y \oplus \sL$.  Let $\phi_0$ be  the endomorphism of $\cE$ that acts by $+1$ on $\cO_\Y$ and by $-1$ on $\sL$, so that $\phi_0$ is traceless.  Therefore the section $\phi = Z \otimes \phi_0 \in \Gamma(\Y,\tshf{\Y}\otimes \cEnd_0(\cE))$
defines a Poisson structure $\pi$ on $\X = \PP(\cE)$.  Since $\cO_\Y$ and $\sL$ are the eigenspaces of  $\phi_0$, the corresponding sections $\bS_0,\bS_1 \subset \X$ give the zero locus of the Poisson structure, i.e. $\Zeros(\pi) = \bS_0 + \bS_1 \subset \X$ is the union of two disjoint copies of $\Y$. \qed
\end{example}

\begin{exercise}
Let $\rho : \X = \PP(\cE) \to \Y$ be a ruled surface, and let $\pi$ be the Poisson structure on $\X$ corresponding to a section $\phi \in \Gamma(\Y,\tshf{\Y}\otimes\cEnd_0(\cE))$.  Let $\D \subset \Y$ be the divisor of zeros of $\phi$.  By considering the relation between the zeros of $\pi$ and the eigenspaces of $\phi$, show that
\[
\Zeros(\pi) = \B + \pi^{-1}(\D)
\]
where $\B \subset \X$ is a divisor that meets each fibre of $\rho$ in a pair of points, counted with multiplicity. \qed
\end{exercise}

As a brief digression from surfaces, let us remark that this construction of Poisson structures on $\PP^1$-bundles can be extended to the higher dimensional setting; see \cite[Section 6]{Polishchuk1997}  for details.  In short, the data one needs are a Poisson structure on the base $\Y$ and a ``Poisson module'' structure on the bundle $\cE$.  The latter is essentially an action of the Lie algebra $(\cO_\Y,\{\cdot,\cdot\})$ on the sections of $\cE$ by Hamiltonian derivations.  When the Poisson structure on $\Y$ is identically zero, it boils down to the following construction:

\begin{exercise}\label{ex:higher-cohiggs}
Let $\Y$ be a manifold, and let $\cE$ be a rank-two bundle on $\Y$.  Arguing as above, we see that a section $\phi \in \Gamma(\Y,\tshf{\Y} \otimes \cEnd_0(\cE))$ defines a bivector field $\pi$ on $\X = \PP(\Y)$, but the integrability condition $[\pi,\pi]=0$ is not automatic if $\dim \Y > 1$.  Show that $\pi$ is integrable if and only if $\phi$ is a \defn{co-Higgs field}~\cite{Hitchin2011,Rayan2011}, meaning that
\[
[\phi \wedge \phi] = 0 \in \wedge^2\tshf{\Y} \otimes \cEnd(\cE),
\]
where $[-\wedge-]$ combines the wedge product $\tshf{\Y} \times \tshf{\Y} \to \wedge^2 \tshf{\Y}$ and the commutator $[-,-]:\cEnd(\cE) \times \cEnd(\cE) \to \cEnd(\cE)$. \qed
\end{exercise}

\subsubsection{Classification of ruled Poisson surfaces}

We are now in a position to state the classification of Poisson structures on ruled surfaces:

\begin{theorem}[Bartocci--Macr{\'{\i}}~\cite{Bartocci2005}, Ingalls~\cite{Ingalls1998}]
If $(\X,\pi)$ is a compact Poisson surface ruled over a curve $\Y$ of genus $g$, then it falls into one of the following classes:
\begin{itemize}
\item $g$ is arbitrary and $(\X,\pi)$ is a compactified cotangent bundle as in \autoref{sec:cotan-compact}
\item $g = 1$, and $(\X,\pi)$ comes from the construction in \autoref{ex:split-ell}.
\item $g = 0$ and $\X = \PP(\cO_{\PP^1} \oplus \cO_\PP^1(k))$ with $k \ge 0$.  See \cite[Lemma 7.11]{Ingalls1998} for a description of the possible anti-canonical divisors in this case.
\end{itemize}
\end{theorem}

We shall not give the full proof here, but let us give an idea of why it is true by explaining one of its corollaries
\begin{corollary}
A Poisson surface $(\X,\pi)$ ruled over a curve $\Y$ of genus at least two must be a compactified cotangent bundle.
\end{corollary}

\begin{proof}
Let $\X = \PP(\cE)$ for a rank-two bundle on $\Y$.  By \autoref{thm:p1bundle}, the Poisson structure corresponds to a section $\phi \in \tshf{\Y}\otimes \cEnd_0(\cE)$.  Let us view $\phi$ as a map
\[
\phi : \cE \to \cE \otimes \tshf{\Y}
\]
Its determinant gives a map
\[
\det \phi : \det \cE \to \det( \cE  \otimes \tshf{\Y}) \cong \det \cE \otimes (\tshf{\Y})^{\otimes 2}
\]
or equivalently a section
\[
\det \phi \in \Gamma(\Y,(\tshf{\Y})^{\otimes 2}).
\]
But for a curve $\Y$ of genus at least two, the bundle $(\tshf{\Y})^{\otimes 2}$ has negative degree, and hence it admits no nonzero sections.  Therefore $\det \phi = 0$ identically, and since $\phi$ is also traceless, it must be nilpotent.

 Now let $\sL \subset \cE$ be the kernel of $\phi$; it is a line subbundle.  (This is obvious away from the zeros of $\phi$, but one can check that it extends uniquely over the zeros as well.)  We have an exact sequence
\begin{align}
\xymatrix{
0 \ar[r] & \sL \ar[r] & \cE \ar[r] & \cE/\sL \ar[r] & 0
}\label{eqn:nilp}
\end{align}
and because of the nilpotence, $\phi$ induces a map $\cE/\sL \to \sL\otimes \tshf{\Y}$ which determines it completely.

Since the projective bundle $\PP(\cE)$ is unchanged if we tensor $\cE$ by a line bundle, we may tensor \eqref{eqn:nilp} by the dual of $\cE/\sL$ and assume without loss of generality that $\cE/\sL \cong \cO_\Y$. Then $\phi$ is determined by  a bundle map $\cO_\Y \to \sL\otimes \tshf{\Y}$  which may vanish on a divisor $\D \subset \Y$.  This gives a canonical identification $\sL \cong \cotshf{\Y}(\D)$, so that $\cE$ fits in an exact sequence
\[
\xymatrix{
0 \ar[r] & \tshf{\Y}(\D) \ar[r] & \cE \ar[r] & \cO_\Y \ar[r] & 0,
}
\]
from which the statement follows easily.
\end{proof}

\subsection{Blowups and minimal surfaces}

We now describe a procedure for producing new Poisson surfaces from old ones, using one of the most basic and important operations in algebraic geometry: blowing up.  We shall briefly recall how blowups of points in complex surfaces work, and refer to \cite[Section 4.1]{Griffiths1994} for details.  We remark that one can also blow up Poisson structures on higher dimensional; see \cite[Section 8]{Polishchuk1997}.

\subsubsection{Blowing up surfaces}

\paragraph{Blowing up and down}
Let us begin by recalling how to blow up a point in a surface, starting with the origin in $\CC^2$.  The idea is to delete the origin, and replace it with a copy of $\PP^1$ that parametrizes all of the lines through this point, as shown in \autoref{fig:blowup}.  One imagines zooming in on the origin in $\CC^2$ so drastically that the lines through the origin become separated.

\begin{figure}[h]
\begin{center}
\begin{tikzpicture}[scale=2]
\draw[clip] (-1,-1) rectangle (1,1);
\draw[fill=lightgrey] (-1,-1) rectangle (1,1);
\draw[thick,red] (60:2) -- (60+180:2);
\draw[thick,purple] (120:2) -- (120+180:2);
\draw[thick,orange] (180:2) -- (180+180:2);
\draw[fill,blue] (0,0) circle (0.04);
\draw (-0.8,0.8) node {$\CC^2$};
\end{tikzpicture}
\hspace{2cm}
\begin{tikzpicture}[scale=2]
\def\r{0.4}
\draw[clip] (-1,-1) rectangle (1,1);
\draw[fill=lightgrey] (-1,-1) rectangle (1,1);
\draw[thick,red] (60:2) -- (60+180:2);
\draw[thick,purple] (120:2) -- (120+180:2);
\draw[thick,orange] (180:2) -- (180+180:2);
\draw[fill,white] (0,0) circle (\r);
\draw[thick,blue,->] (-30:\r) arc (-30:150:\r);
\draw[thick,blue,->] (150:\r) arc (150:330:\r);
\draw[red,fill] (60:\r) circle (0.02);
\draw[red,fill] (180+60:\r) circle (0.02);
\draw[purple,fill] (120:\r) circle (0.02);
\draw[purple,fill] (180+120:\r) circle (0.02);
\draw[orange,fill] (0:\r) circle (0.02);
\draw[orange,fill] (180:\r) circle (0.02);
\draw (-0.8,0.8) node {$\widetilde{\CC}^2$};
\draw[blue] (0.65,0.3) node {$\E \cong \PP^1$};
\end{tikzpicture}
\end{center}
\caption{The blowup of the origin in $\CC^2$ replaces the origin with a copy of $\PP^1$ called the exceptional curve.  The lines through the origin correspond to disjoint curves in the blowup, each meeting the exceptional curve in a single point.}
\label{fig:blowup}
\end{figure}
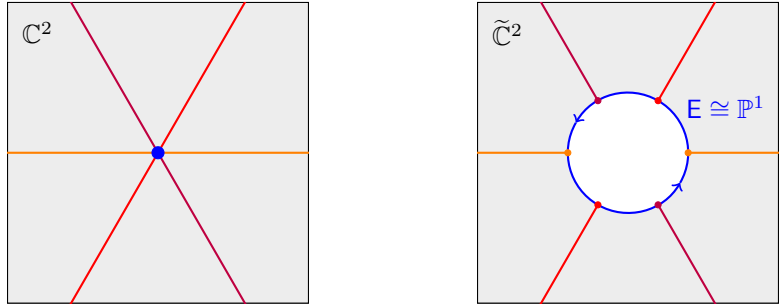

To be more precise, the \defn{blowup} of $\CC^2$ at the origin is the set $\widetilde{\CC}^2$ consisting of pairs $(p,\bL)$, where $\bL \subset \CC^2$ is a one-dimensional linear subspace, and $p \in \bL$.  There is a natural map
\[
b : \widetilde{\CC}^2 \to \CC^2,
\]
called the \defn{blowdown}, which simply forgets the line $\bL$:
\[
b(p,\bL) = p.
\]
If $p \ne 0 \in \CC^2$, then there is a unique linear subspace through $p$, so $b^{-1}(p)$ is a single point.  On the other hand, if $p=0$, then every subspace contains $p$, and we get
\[
b^{-1}(p) = \{\textrm{lines through 0 in }\CC^2\} = \PP^1
\]
So $\widetilde{\CC}^2$ is obtained by replacing the origin in $\CC^2$ with a copy of $\PP^1$, which we will henceforth refer to as the \defn{exceptional divisor}
\[
\E \cong \PP^1 \subset \widetilde{\CC}^2.
\]
To see that the blowup is actually a complex surface, consider the map $\widetilde{\CC}^2 \to \PP^1$ given by $(p,\bL) \mapsto \bL$.  The fibre over $\bL \in \PP^1$ is evidently  a copy of the one-dimensional vector space $\bL$.  Thus $\widetilde{\CC}^2$ is, in a natural way, the total space of a line bundle over $\PP^1$---the so-called \defn{tautological line bundle} $\cO_{\PP^1}(-1)$.  From this perspective, the exceptional curve is the zero section of $\cO_{\PP^1}(-1)$. 

More concretely, we may define coordinates on $\widetilde{\CC}^2$ using an affine coordinate chart $\set{[1:v]}{v\in \CC} \subset \PP^1$.  Every point $p$ on the line  $\bL = [1:v] \in \PP^1$ has the form $p = (u,uv)$ for some unique $u \in \CC$.  This gives coordinates $(u,v)$ on the open set
\[
\U = \set{(p,\bL) \in \widetilde{\CC^2} }{\bL\textrm{ is not the }y\textrm{-axis}}
\]
In these coordinates, the blowdown is given by
\[
b(u,v) = (u,uv)
\]
and the exceptional curve $\E\cap \U \subset \U$ is defined by the equation $u=0$.  Clearly similar formulae are valid in the other chart on $\PP^1$.

\paragraph{Blowing up points on abstract surfaces}
The local picture above can be replicated on any complex surface. Let $p \in \X$ be a point in complex surface.  While there is no reasonable notion of a one-dimensional linear subspace through $p$ in $\X$, we can instead consider lines in the tangent space $\T_p\X$.  The set of one-dimensional subspaces of $\T_p\X$ forms a projective line
\[
\E = \PP(\T_p\X) \cong \PP^1.
\]
The \defn{blowup of $\X$ at $p$} is then the surface $\tX$ that is obtained by replacing $p$ with the curve $\E$.  Thus, as a set, we have
\[
\tX = \rbrac{\X \setminus \{p\}} \sqcup \E.
\]
The blow-down map
\[
b : \tX \to \X
\]
is the identity map on $\X\setminus \{p\}$, but contracts the whole exceptional curve to $p$:
\[
b(\E) = \{p\}
\]
Using local coordinates at $p$, our calculations on $\CC^2$ above can be used to give $\tX$ the structure of a complex surface.  Thus a tubular neighbourhood of $\E \cong \PP^1$ in $\tX$ is isomorphic to a neighbourhood of the zero section in the line bundle $\cO_{\PP^1}(-1)$.  This means that the normal bundle of $\E$ is a copy of $\cO_{\PP^1}(-1)$, and that the self-intersection number of $\E$ is 
\[
[\E]\cdot[\E] = - 1.
\]

The surface $\tX$, being isomorphic to $\X$ away from $p$, is only slightly more complicated than $\X$ itself.  For example, one can use a  Mayer--Vietoris argument to compute the homology groups:
\begin{proposition}\label{prop:blowup-hlgy}
If $\tX$ is the blowup of $\X$, with exceptional curve $\E \subset \tX$, then its homology groups are given by
\[
\hlgy[i]{\tX,\ZZ} = \begin{cases}
\hlgy[i]{\X,\ZZ} & i \ne 2 \\
\hlgy[2]{\X,\ZZ} \oplus \ZZ \cdot [\E] & i =2.
\end{cases}
\]
\end{proposition}

\paragraph{Blowing down}

There is also a method for deciding when a given surface $\Y$ can be obtained as the blowup of another surface.  The idea is to characterize the  curves $\E  \subset \Y$ that are candidates for the exceptional divisor of a blowup.    First of all, by construction, such a curve must be isomorphic to $\PP^1$.  The standard nomenclature for such a curve is as follows:

\begin{definition}
A \defn{smooth rational curve} is a complex manifold that is isomorphic to the projective line $\PP^1$.
\end{definition}

As we have seen, the rational curves that arise as exceptional curves in a surface are embedded in a special way: they have a tubular neighbourhood isomorphic to the zero section in the tautological line bundle over $\PP^1$.  We thus single out these curves as special:

\begin{definition}
A \defn{$(-1)$-curve} on a surface $\Y$ is a smooth rational curve $\E \subset \Y$ having one (and hence all) of the following equivalent properties:
\begin{itemize}
\item The degree of the normal bundle of $\E$ is equal to $-1$.
\item The normal bundle of $\E \cong \PP^1$ is isomorphic to the tautological line bundle $\cO_{\PP^1}(-1)$.
\item The self-intersection number of $\E$ is $[\E]\cdot[\E] = -1$.
\end{itemize}
\end{definition}

The following important result explains that these criteria completely characterize the curves that arise as exceptional divisors of blowups:
\begin{theorem}[Castelnuovo--Enriques]
Let $\Y$ be a complex surface, and suppose that $\E \subset \Y$ is a $(-1)$-curve.  Then there is a surface $\X$ and a point $p \in \X$ such that $\Y$ is isomorphic to the blowup $\tX$ of $\X$ at $p$, and $\E$ is identified with the exceptional curve in $\tX$.
\end{theorem}

\begin{proof} See \cite[Section 4.1]{Griffiths1994}.
\end{proof}

\subsubsection{Blowing up Poisson brackets}
\subsubsection*{Blowing up}
Suppose now that the surface $\X$ carries a Poisson structure $\pi$.  If $p \in \X$, we can form the blowup $\tX$ and we might ask whether $\tX$ inherits a Poisson structure as well.  More precisely, we have the blowdown map
\[
b : \tX \to \X
\]
and we ask whether there exists a Poisson structure on $\tX$ such that $b$ is a Poisson map, i.e.~we ask that
\[
b^*\{f,g\} = \{b^*f,b^*g\}
\]
for all functions $f,g \in \cO_\X$.  Since $b$ is an isomorphism over an open dense set, there can be at most one Poisson structure on $\tX$ with this property.  Hence, the question is simply whether the Poisson structure on $\tX \setminus \E \cong \X\setminus \{p\}$ may be extended holomorphically to $\tX$.

Clearly, the answer depends only on the local  behaviour of the Poisson structure in a neighbourhood of $p$, so we can work in local coordinates $x,y$ centred at $p$, and write
\[
\{x,y\} = f(x,y)
\]
for a holomorphic function $f$.

Let us choose corresponding coordinates $u,v$ on the blowup as above, so that
\begin{align*}
b^*x &= u & b^*y &= uv
\end{align*}
In particular,
\[
v = b^*(x^{-1}y)
\]
In order for the bracket to extend to the blowup, we need to ensure that $\{u,v\}$ is holomorphic in the whole coordinate chart $u,v$.  We compute
\begin{align*}
\{u,v\} &= \{b^*x,b^*(x^{-1}y)\} \\
&= b^*\{x,x^{-1}y\} \\
&= b^*(x^{-1}f(x,y)) \\
&= u^{-1}f(u,uv) \\
&= u^{-1}(f(0,0)+ug(u,v))
\end{align*}
where $g$ is holomorphic near the locus $u=0$, i.e.~the along the exceptional curve $\E$.  So in order for $\{u,v\}$ to be holomorphic along $\E$, it is necessary and sufficient that $f(0,0)= 0$.  In other words, the bivector $\pi$ on $\X$ must vanish at the point $p$ that we have blown up.  Evidently, the calculation is identical in the other coordinate chart on the blowup, and so we arrive at the

\begin{proposition}
Let $(\X,\pi)$ be a Poisson surface, let $p \in \X$ be a point in $\X$ and let $\tX$ be the blowup of $\X$ at $p$.  Then $\tX$ carries a Poisson structure $\widetilde{\pi}$ such that the blowdown map $b :\tX \to \X$ is Poisson if and only if $\pi$ vanishes at $p$.
\end{proposition}

\begin{exercise}
Amongst all the possible singularities of a  curve in $\CC^2$, there are three special classes called $A$, $D$ and $E$---the simple singularities~\cite{Arnold1972}.  They are the zero sets of the polynomials in the following table:
\begin{center}
\begin{tabular}{c|c|c|c|c}
$A_k$, $k \ge 1$ & $D_k$, $k \ge 4$ & $E_6$ & $E_7$ & $E_8$ \\\hline
\rule{0pt}{2.6ex}
 $x^2+y^{k+1}$  & $x^2y+y^{k-1}$ & $x^3+y^4$ & $x^3+xy^3$  & $x^3+y^5$
\end{tabular}
\end{center}
Let $f$ be one of these polynomials, and define a Poisson structure $\pi$ on $\CC^2$ by
\[
\pi = f \cvf{x}\wedge \cvf{y}.
\]
Let $\widetilde \pi$ be the Poisson structure obtained by blowing up $\pi$ at the origin in $\CC^2$.  Describe the divisor $\D \subset \widetilde{\CC}^2$ on which $\widetilde \pi$ vanishes.\qed
\end{exercise}

\subsubsection*{Blowing down}

While blowing up Poisson structures requires some care, blowing them down is much easier.  In fact, we have the following general result, observed in \cite[Proposition 8.4]{Polishchuk1997}, which shows that holomorphic Poisson structures can often be pushed forward along maps.  Note that the analogous statement for $C^\infty$ manifolds fails dramatically.

\begin{proposition}
Let $\X$ be a complex Poisson manifold, and let $\phi : \X \to \Y$ be a holomorphic map satisfying one of the following conditions:
\begin{enumerate}
\item $\phi$ is surjective and proper with connected fibres; or
\item $\phi$ is an isomorphism away from an analytic subset $\Z \subset \Y$ of codimension at least two.
\end{enumerate}
Then there is a unique Poisson structure on $\Y$ such that $\phi$ is a Poisson map.
\end{proposition}

\begin{proof}
Suppose given an open set $\U \subset \Y$ and its preimage $\phi^{-1}(\U) \subset \X$.  We need to define a Poisson bracket on $\U$ such that the pullback map 
\[
\phi^*: \Gamma(\U,\cO_\U) \to \Gamma(\phi^{-1}(\U),\cO_{\phi^{-1}(\U)})
\]
is a morphism of Poisson algebras.

We claim that, under our assumptions, $\phi^*$ is already an isomorphism of algebras, so that this is immediate.  Indeed, in the first case, the restriction of any global function $f$ on $\phi^{-1}(\U)$ to a fibre is evidently constant, so that $f$ is the pullback of a function on $\Y$.  In the second case, the isomorphism is a consequence of Hartogs' theorem, which implies that holomorphic functions have unique extensions over codimension two subsets~~\cite[Section 0.1]{Griffiths1994}.
\end{proof}

\begin{remark}
What we have really used is the fact that we have an isomorphism $\phi_*\cO_\X \cong \cO_\Y$ of sheaves on $\Y$.\qed
\end{remark}

\begin{corollary}
Let $\X$ be a complex surface and $b: \tX \to \X$ be its blowup at a point.  Then for any Poisson structure on $\tX$, there is a unique Poisson structure on $\X$ such that $b$ is a Poisson map.
\end{corollary}

\subsection{Birational classification of Poisson surfaces}

The fact that we can always blow up points on surfaces means that classifying all surfaces up to isomorphism is likely an intractable task. A more reasonable goal is to find a list of ``minimal'' surfaces---surface which are not blowups of other surfaces---and describe the others in terms of these.
 
\begin{definition}
A compact complex surface is \defn{minimal} if it contains no $(-1)$-curves.
\end{definition}

Every non-minimal surface can be obtained from a minimal one by a sequence of blowups.  Indeed, suppose that $\X$ is a compact surface that is not minimal.  Then $\X$ contains a $(-1)$-curve which we may blow down to obtain a surface $\X_1$.  Then, if $\X_1$ contains a $(-1)$-curve, we can blow it down to get a surface $\X_2$, and so on.  Continuing in this way, we produce a sequence of surfaces
\[
\X \to \X_1 \to \X_2 \to \cdots ,
\]
each obtained by blowing down the previous one.  We claim that this process must halt after finitely many steps,  yielding a minimal surface $\X_n$.  Indeed, one way to see this is to recall from \autoref{prop:blowup-hlgy} that blowing down decreases the rank of the second homology group.  So the rank of the second homology of $\X$ gives an upper bound on the number of blowdowns that we could possibly perform.

One of the major results of 20th century algebraic and complex geometry was a coarse classification of the minimal surfaces into 10 distinct types according to various numerical invariants such as Betti numbers---the \defn{Enriques--Kodaira classification}.  A  discussion of these results would take us too far afield; a detailed treatment can be found, for example in \cite{Barth2004} or \cite[Section 4.5]{Griffiths1994}.  We describe here the classification of minimal \emph{Poisson} surfaces.  They fall into two broad classes: the symplectic surfaces, and the degenerate ones.

\subsubsection{Symplectic surfaces}
\label{sec:symp-surf}

Symplectic surfaces are surfaces equipped with nonvanishing Poisson structures.  In other words, a symplectic surface $\X$ admits a nonvanishing section of its anticanonical line bundle.  This means that the anticanonical bundle is trivial, and the space of Poisson structures on $\X$ is one-dimensional; they are all constant multiples of one another. 

Any compact symplectic surface is automatically minimal.  Moreover, since the Poisson structure is nonvanishing, it cannot be blown up to obtain a new Poisson surface.  The Enriques--Kodaira classification gives a complete list of symplectic surfaces:
\begin{theorem}
A surface $\X$ is symplectic if and only if it is either a complex torus or a K3 surface.
\end{theorem}

\paragraph{Complex tori:} These are the surfaces that are isomorphic to a quotient $\CC^2 / \Lambda$, where $\Lambda \cong \ZZ^4 \subset \CC^2$ is a lattice of translations; hence they are topologically equivalent to a four-torus, but different lattices will result in non-isomorphic complex structures.  These surfaces are symplectic because the standard Darboux symplectic structure on $\CC^2$ is invariant under translation, and hence it descends to the quotient.  A complex torus that admits an embedding in projective space is called an \defn{abelian variety}; these are characterized by  the classical Riemann bilinear relations~\cite[Section 2.6]{Griffiths1994}.

\paragraph{K3 surfaces:} These are the compact symplectic surfaces that are simply connected.  They are all diffeomorphic as $C^\infty$ manifolds, but there are many non-isomorphic complex structures in this class.  One way to produce a K3 surface is to take the zero locus in $\PP^3$ of a homogeneous quartic polynomial, but there are many examples that do not arise in this way; indeed, many K3 surfaces cannot be embedded in any projective space.   See \cite{Huybrechts2016} for a comprehensive treatment of these surfaces.

\subsubsection{Surfaces with degenerate Poisson structures}

It remains to deal with the case of \defn{degenerate} Poisson structures---Poisson structures whose divisor of zeros is nonempty.  The Enriques--Kodaira classification gives the following list of possibilities:

\begin{theorem}
Let $\X$ be a minimal surface that admits a degenerate Poisson structure.  Then $\X$ is either $\PP^2$, a ruled surface, or a class \classseven{} surface.
\end{theorem}

\subparagraph{The projective plane:} We have already seen the classification of Poisson structures on $\PP^2$ in \autoref{sec:p2}; they are essentially the same as cubic curves.  

\subparagraph{Ruled surfaces:} We dealt with the classification of these in \autoref{sec:ruled}.

\subparagraph{Class  \classseven{}  surfaces:} These are the minimal surfaces whose first Betti number is equal to one, and for which the powers of the canonical bundle $\can_\X^{d}$ for $d> 0$ have no nonzero holomorphic sections.  These surfaces do not admit K\"ahler metrics, and in particular, they are not projective.

At the time of writing, the classification of class \classseven{} surfaces remains a major open problem in the theory of complex surfaces.  However, the list of class \classseven{} Poisson surfaces can be assembled from known results.  Each such surface contains a so-called \defn{global spherical shell}---an open subset that is  isomorphic to a tubular neighbourhood of the three-sphere in $\CC^2$, and whose complement is connected.   As explained in \cite{Dloussky1984}, the geometry of these surfaces is intimately connected with the behaviour of germs of mappings $(\CC^2,0) \to (\CC^2,0)$.  We thank Georges Dloussky for providing us with the following classification of pairs $(\X,\D)$ where $\X$ is a class \classseven{} surface and  $\D \subset \X$ is an anticanonical divisor:
\begin{itemize}
\item $\X$ is a \defn{Hopf surface}.  Such a surface is either diffeomorphic to the product $S^1\times S^3$, in which case it is called \defn{primary}, or it is a finite quotient of a primary Hopf surface.  The anticanonical divisor is either a disjoint pair of elliptic curves with multiplicity one, or a single elliptic curve with positive multiplicity.  See \cite[Section 10]{Kodaira1966} for an introduction to Hopf surfaces and their anticanonical divisors, and \cite{Kato1975a,Kato1989} for a discussion of the possible quotients of the primary ones.  (Notice that not every quotient is Poisson, since the group action may not preserve the Poisson structure on the primary surface.)
\item $\X$ is a \defn{parabolic Inoue surface} and $\D$ is the disjoint union of an elliptic curve and a cycle of rational curves.  Such surfaces are examples of Enoki surfaces~\cite{Enoki1981}.
\item $\X$ is a \defn{hyperbolic Inoue surface} (also known as an even Inoue--Hirzebruch surface), and   $\D$ is the sum of two disjoint cycles of rational curves; see \cite[p.~103]{Inoue1977} and \cite[Proposition 2.14]{Dloussky1988}.
\item $\X$ is an \defn{intermediate Kato surface}, and belongs to a special hypersurface in the moduli space of such surfaces; see \cite[Theorem 5.2]{Dloussky1999} and \cite[Proposition 4.24]{Dloussky2014}.  There is a finite collection $\D_1,\ldots,\D_n$ of rational curves in $\X$ whose fundamental classes generate $\hlgy[2]{\X,\QQ}$.  Each of these curves is either smooth or nodal, and every divisor in $\X$ is a linear combination of them.  The existence of an anticanonical divisor in $\X$ depends on the existence of a solution to a linear system of Diophantine equations, defined using the intersection form on $\hlgy[2]{\X,\ZZ}$ and the arithmetic genera of the components~\cite[Lemma 4.2]{Dloussky1999}.
\end{itemize}

\section{Poisson threefolds}

\label{sec:threefolds}

We now turn our attention to three-dimensional Poisson structures.  Dimension three is the lowest in which the integrability condition $[\pi,\pi] = 0$ for a Poisson structure is nontrivial.  Correspondingly, there is a substantial increase in complexity compared with Poisson surfaces.

If $\X$ is a threefold, we have an isomorphism
\[
\wedge^2\tshf{\X} \cong \forms[1]_\X \otimes \wedge^3 \tshf{\X} = \forms[1]_\X \otimes \acan_\X,
\]
so that bivectors can be alternatively be viewed as one-forms with values in the anticanonical line bundle.  In other words, every bivector field may be written locally as an interior product
\[
\pi = \hook{\alpha}\mu
\]
where $\mu \in \acan_\X$ and $\alpha \in \forms[1]_\X$.  One can easily check that the integrability condition $[\pi,\pi] = 0$ is equivalent to the equation
\begin{align}
\alpha \wedge d \alpha = 0 \label{eqn:int-oneform}
\end{align}
that ensures that the kernel of $\alpha$ gives an integrable distribution on $\X$.  

\begin{exercise}
Verify this claim. \qed
\end{exercise}

As was the case for surfaces, the symplectic leaves must all have dimension zero or two.  But now the  two-dimensional symplectic leaves are no longer open, and their behaviour can be quite complicated; for example, the individual leaves can be dense in $\X$.  Nevertheless, one can get some very good control over the behaviour and classification of Poisson threefolds.

\subsection{Regular Poisson structures}

As a warmup, let us consider the simplest class of Poisson threefolds: the regular ones.  We recall that a Poisson manifold  $(\X,\pi)$ is  \defn{regular} if all of its symplectic leaves have the same dimension.  Equivalently, $\pi$ is regular if it has constant rank, when viewed as a bilinear form on the cotangent spaces of $\X$.

For a nonzero Poisson structure on a threefold, regularity means that all of the leaves have dimension two, and a theorem of Weinstein~\cite{Weinstein1987} implies that the Poisson structure is locally equivalent to a product $\CC^2 \times \CC$ where $\CC^2$ has the standard Poisson structure $\cvf{x}\wedge \cvf{y}$ and $\CC$ carries the zero Poisson structure.

Clearly, if $\Y$ is a symplectic surface and $\Z$ is a curve, the product $\X = \Y \times \Z$ carries a regular Poisson structure whose symplectic leaves are the fibres of the projection to $\Z$.  Now suppose that $\X$ carries a free and properly discontinuous action of a discrete group $\G$, and that the Poisson structure is invariant under the action of $\G$.  Then the quotient $\X/\G$ will be a new Poisson threefold, and since the quotient map $\X \to \X/\G$ is a covering map compatible with the Poisson brackets, it follows that the Poisson structure on $\X/\G$ is regular.

In this way, one can easily construct examples of compact Poisson threefolds whose individual symplectic leaves are dense submanifolds:

\begin{exercise}
Consider the Poisson structure on $\CC^3 \cong \CC^2 \times \CC$ given in coordinates $x,y,z$ by
\[
\pi = \cvf{x}\wedge\cvf{y}.
\]
Let $\Lambda \cong \ZZ^6 \subset \CC^3$ be a generic lattice of translations, so that $\X = \CC^3/\Lambda$ is a compact six-torus.  Determine the conditions under which the symplectic leaves of $\X$ will be dense. \qed
\end{exercise}

The construction of regular Poisson threefolds from discrete group actions may seem somewhat simplistic, but in fact all regular projective Poisson threefolds arise in this way.  This is guaranteed by the following remarkable and nontrivial theorem of Druel, which relies on results from the minimal model program for threefolds:  

\begin{theorem}[\cite{Druel1999}]\label{thm:druel}
Suppose that $(\X,\pi)$ is a smooth projective Poisson threefold, and that the zero set $\Zeros(\pi)\subset \X$ is finite.  Then in fact $\Zeros(\pi)$ is empty, so that $\pi$ regular.  Moreover $(\X,\pi)$ is isomorphic to a quotient
\[
\X \cong (\Y \times \Z)/\G
\]
as above, and it falls into one of the following four classes:
\begin{itemize} 
\item $\Y = \CC^2$ with the standard Poisson structure, and $\Z=\CC$.  The group $\G$ is a lattice of translations on $\CC^2\times \CC$, so that $\X$ is an abelian threefold, and the symplectic leaves are given by (possibly irrational) linear flows.
\item $\Y = \CC^2$ and $\Z = \PP^1$.  The group $\G$ is a lattice of translations on $\Y$, which also acts on $\Z$  by projective transformations.  Thus $\X$ is a flat $\PP^1$-bundle over an abelian surface; the symplectic leaves are the horizontal sections of the flat connection.
\item $\Y$ is an abelian surface and $\Z$ is a compact curve.  The group $\G  \subset \Aut(\Z)$ acts on $\Y \times \Z$ by
\[
g \cdot (y,z) = (u_g(y)+t_g(z), g\cdot z)
\]
where $u_g \in \Aut(\Y)$ is a symplectic automorphism of $\Y$ that preserves both its group structure, and 
 $t_g : \Z \to \Y$ is a holomorphic map.  Thus $\X$ is a symplectic fibre bundle with abelian fibres over an orbifold curve.

\item $\Y$ is a K3 surface and $\Z$ is a compact curve. The action of $\G$ on $\Y\times \Z$ is the product of a free action on $\Z$ and an action on $\Y$ that preserves its symplectic structure.  Thus $\X$ is a symplectic fibre bundle with K3 fibres over a smooth curve.
\end{itemize}
\end{theorem}

\subsection{Poisson structures from pencils of surfaces}

We now turn to the case in which the Poisson structure is no longer regular.   As in the regular case, the global behaviour of the two-dimensional leaves may be quite complicated.  But now there is a second source of difficulty: the Poisson structure may exhibit very complicated \emph{local} behaviour, due to the singularities of the foliation in the neighbourhood of the zero-dimensional leaves.

In this subsection, we consider the special case in which the symplectic leaves lie in the level sets of a (possibly meromorphic) function, beginning with the local case $\X = \CC^3$.

\subsubsection{Jacobian Poisson structures  on $\CC^3$}
\label{sec:jacobian}
Let $x,y,z$ be global coordinates on $\CC^3$, and let $f \in \cO_{\CC^3}$ be a nonconstant holomorphic function.  Define a bivector by the formula
\begin{align*}
\pi &= \hook{df}\rbrac{\cvf{x}\wedge\cvf{y}\wedge\cvf{z}} 
\end{align*}
Since $df$ is closed, this bivector is a Poisson structure by \eqref{eqn:int-oneform}.  The elementary Poisson brackets are given by
\begin{align*}
\{x,y\} &= \pderiv{f}{z} & 
\{y,z\} &= \pderiv{f}{x} & 
\{z,x\} &= \pderiv{f}{y}.
\end{align*}
Such a Poisson bracket is called a \defn{Jacobian Poisson structure} because of its link with the derivatives of $f$.

Since the Hamiltonian vector field of $f$ is given by
\[
H_f = \hook{df}\pi = \hook{df}\hook{df}\rbrac{\cvf{x}\wedge\cvf{y}\wedge\cvf{z}} = 0
\]
we must have that $\{f,g\} = 0$ for all functions $g$, i.e.~$f$ is a \defn{Casimir function}.  Put slightly differently, we have $\lie_{H_g}(f) = 0$, so that $f$ is invariant under all Hamiltonian flows.

Since the Hamiltonian flows sweep out the symplectic leaves, it follows that for each $c \in \CC$, the  fibre
\[
\Y_c = f^{-1}(c) \subset \CC^3
\]
is a union of symplectic leaves.  Because of the dimension, each $\Y_c$ is a surface, and we can explicitly describe the  symplectic leaves in terms of the geometry of the surfaces, as follows.

First of all, notice that the points in $\CC^3$ where the Poisson structure vanishes are precisely the points where $df = 0$, i.e.~the critical points of $f$, or equivalently the singular points of the fibres $\Y_c$ for $c \in \CC$.  Away from these points, the symplectic leaves have dimension two---the same dimension as the fibres in which they are contained.  Hence they must be open subsets of the fibres.  We therefore arrive at the following description of the leaves:
\begin{itemize}
\item The zero-dimensional leaves are the singular points of the fibres of $f$
\item The two-dimensional leaves are  the connected components of the smooth loci of the fibres of $f$.
\end{itemize}

\begin{example}\label{ex:sl2}
Let $f = \tfrac{1}{2}x^2+2yz$.  Then we obtain the linear Poisson brackets
\begin{align*}
\{x,y\} &= 2y  & \{y,z\} &= x  & \{z,x\} &= 2z
\end{align*}
that are associated with the $\sln{2,\CC}$ Lie algebra.  The only critical point of $f$ is the origin, where $f$ has a Morse-type singularity.  The level sets of $f$ are the quadric surfaces
\[
\tfrac{1}{2}x^2+2yz = c.
\]
For $c \ne 0$, the level sets are smooth, giving symplectic leaves.  But when $c=0$, the level set is a cone with a singularity at the origin.  It is the union of two leaves: the origin is a zero-dimensional leaf, and the rest of the cone is a two-dimensional leaf.  See \autoref{fig:sl2}. \qed
\end{example}

\begin{figure}
\center
\includegraphics[scale=0.25]{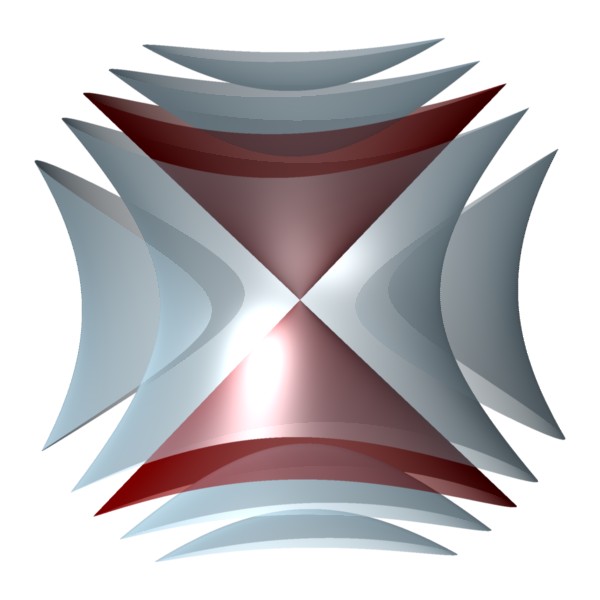}
\caption{Symplectic leaves of the Jacobian Poisson structure on $\CC^3$ defined by the function $f=\tfrac{1}{2}x^2+2yz$.  The red cone is the level set $f^{-1}(0)$, whose singular point is the unique zero-dimensional leaf.}
\label{fig:sl2}
\end{figure}

\begin{example}
Let $f=xyz$, so the Poisson brackets are
\begin{align*}
\{x,y\} &= xy  & \{y,z\} &= yz  & \{z,x\} &= zx
\end{align*}
Let us determine the structure of the level sets
\[
xyz = c.
\]
When $c\ne 0$, all three of $x,y,z$ must be nonzero.  If we fix $x,y \in \CC^*$, then $z$ is uniquely determined as $z = \frac{c}{xy}$.  Thus the level set is smooth, giving a symplectic leaf isomorphic to $(\CC^*)^2$. 

On the other hand, the zero level set is given by the equation
 \[
xyz=0,
\]
and is therefore the union of the coordinates planes in $\CC^3$.  This variety is singular where the planes meet.  Thus the zero-dimensional leaves are given by the coordinate axes in $\CC^3$.  Meanwhile there are three  distinct two-dimensional leaves in this fibre, given by taking each plane and removing the corresponding axes.  This example is shown in \autoref{fig:xyz}
\end{example}

\begin{figure}
\begin{center}
\includegraphics[scale=0.25]{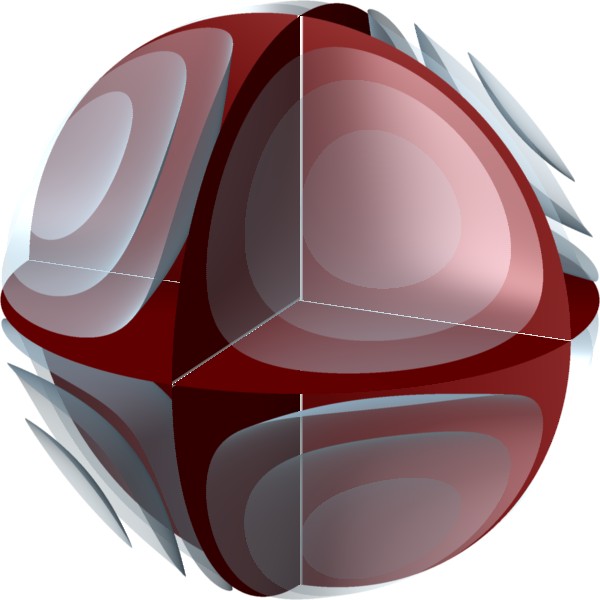}
\end{center}
\caption{Symplectic leaves of the Jacobian Poisson structure on $\CC^3$ defined by the function $f=xyz$.   The zero-dimensional symplectic leaves are the points on the coordinate axes.  The two-dimensional leaves are the coordinate planes minus their axes (shown in red), and the nonzero level sets of $f$  (shown in blue).}\label{fig:xyz}
\end{figure}

\subsubsection{Pencils of symplectic leaves}
\label{sec:pencils}
If $\X$ is a compact threefold, it will not admit any nonconstant global holomorphic functions, so the construction of Jacobian Poisson structures above will not produce anything nontrivial.  However, at least if $\X$ is projective, it will admit many nonconstant meromorphic functions, and we may try to use those instead.  To do so we need to recall another key algebro-geometric notion: that of a pencil of hypersurfaces (see \cite[Section 1.1]{Griffiths1994}).

A meromorphic function can be written locally as
\[
f = \frac{g}{h}
\]
where $g$ and $h$ are holomorphic functions.  Notice that $f$ takes on a well-defined value in $\PP^1 = \CC \cup \{\infty\}$ only at the points where $g$ and $h$ do not both vanish.  By removing any common factors of $g$ and $h$, we can assume that this indeterminacy locus $\B \subset \X$ is either empty or has codimension two in $\X$.  It is called the \defn{base locus of $f$}.  We typically write
\[
f : \X \dashrightarrow \PP^1
\]
to indicate that the map $f$ is only well-defined away from its base locus (which we leave implicit).   It is an example of the more general notion of a \defn{rational map}.

Given a point $t \in \PP^1$, we can take the fibre of $f$ in $\X\setminus \B$.  Its closure is a hypersurface $\D_t \subset \X$, and the base locus is
\[
\B = \bigcap_{t \in \PP^1} \D_t.
\]

\begin{definition}
Let $\X$ be a complex manifold.  A \defn{pencil of hypersurfaces in $\X$} is a family of closed hypersurfaces $\D_t \subset \X$ for $t \in \PP^1$ obtained as the fibres of a meromorphic function as above.
\end{definition}
\autoref{fig:pencil} shows a typical example of a pencil of surfaces.

\begin{figure}[t]
\begin{center}
\includegraphics[scale=0.3]{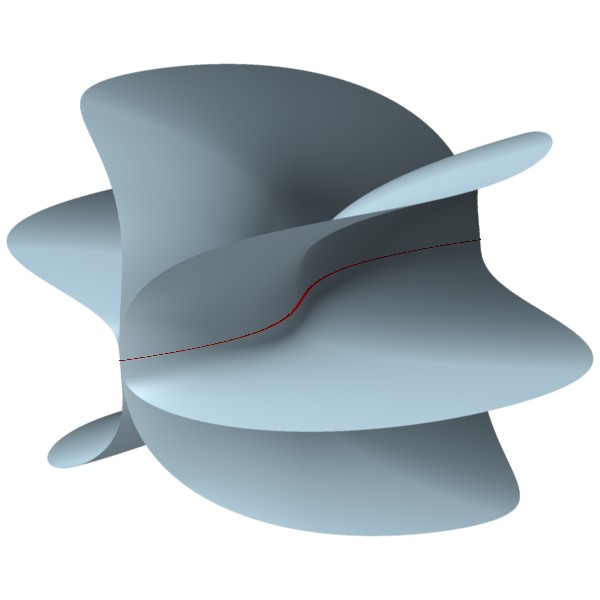}
\end{center}
\caption{A pencil of surfaces, with the base locus shown in red.}\label{fig:pencil}
\end{figure}

Given  $f : \X \dashrightarrow \PP^1$ where $\X$ is a threefold, we may try to repeat the construction of the Jacobian Poisson structure in the previous section as follows.  Thinking of $f$ as a meromorphic function, consider the meromorphic one-form
\[
\alpha = d \log f = \frac{df}{f}.
\]
We would like to contract this one-form into a trivector in order to obtain a Poisson structure.

In order for this to work, we have to deal with the poles of $\alpha$.  Considering a local presentation $f= \frac{g}{h}$, let us factor $g$ and $h$ into irreducible factors
\begin{align*}
g &= g_0^{j_1}\cdots g_m^{j_m} & h &= h_1^{k_1}\cdots k_n^{k_n},
\end{align*}
so that the irreducible components $\D_0$ and $\D_\infty$ and their multiplicities become apparent.  Then we have
\[
\alpha = j_1 \frac{dg_1}{g_1} + \cdots + j_m \frac{dg_m}{g_m} - k_1 \frac{d h_1}{h_1} - \cdots - k_n \frac{d h_n}{h_n}.
\]
It is now clear that, even if the fibres $\D_0$ and $\D_\infty$ have components of high multiplicity, the one-form $\alpha$ has only first-order poles.  We conclude that the polar divisor of $\alpha$ is the ``reduced divisor''
\[
\D = (\D_0 + \D_\infty)_{\textrm{red}} \subset \X,
\]
which has all the same irreducible components as $\D_0 + \D_\infty$, but with all multiplicities set equal to one.

Now suppose that $\D$ is an anticanonical divisor, cut out by a section $\mu \in \Gamma(\X,\acan_\X)$.  If we form the contraction
\[
\pi = \hook{\alpha} \mu,
\]
the zeros of $\mu$ will exactly cancel the poles of $\alpha$, and we will obtain a globally defined \emph{holomorphic} bivector field on $\X$.   Since $\alpha$ is closed, this bivector satisfies $[\pi,\pi] = 0$ and hence we obtain a Poisson structure on $\X$.  

The symplectic leaves of $\pi$ can now be described fairly easily.  Each hypersurface $\D_t\subset \X$ for $t \in \PP^1$ is a union of symplectic leaves, so every two-dimensional leaf is an open subset of some $\D_t$.  Meanwhile, every point in the base locus $\B = \bigcap_{t \in \PP^1} \D_t$ is a zero dimensional leaf, as are the singular points of the fibres.

\begin{exercise}\label{ex:local-pencil}
Suppose that $p \in \B$ is a point of the base locus at which $\D_0$ and $\D_\infty$ are smooth and transverse.  Show that in a suitable system of coordinates $x,y,z$ on $\X$ near $p$, the Poisson structure has the form
\[
\pi = (x\cvf{x}+y\cvf{y})\wedge\cvf{z}.
\]
Give equations for its symplectic leaves. \qed
\end{exercise}

This construction of Poisson structures from pencils  can be generalized in two ways: firstly, we can allow the possibility that $\alpha$ has zeros as well as poles, in which case we can allow the section of $\acan_\X$ to have poles as well.  Secondly, we can consider maps to curves of positive genus instead of $\PP^1$; see \cite[Section 13]{Polishchuk1997} for details.

\subsubsection{Poisson structures from pencils on $\PP^3$}

Let us now describe some examples of the above construction in the case where $\X = \PP^3$.  First of all, we note that every pencil on $\PP^3$ has the form
\[
f([x_0:x_1:x_2:x_3]) = \frac{G(x_0,x_1,x_2,x_3)}{H(x_0,x_1,x_2,x_3)}
\]
where $G$ and $H$ are homogeneous polynomials of the same degree. Evidently  the fibres over $0$ and $\infty$ are given by
\begin{align*}
\D_0 &= \Zeros(G) & \D_\infty = \Zeros(H).
\end{align*}
Now the anticanonical divisors in $\PP^3$ are precisely the quartic divisors, just as anticanonical divisors in $\PP^2$ were cubics.  We therefore we try to arrange so that $(\D_0+\D_\infty)_\red$ is a quartic.

\begin{example}[Sklyanin {\cite{Sklyanin1982}}]\label{ex:sklyanin}  Suppose that $G$ and $H$ are homogeneous quadratic polynomials.  Then the surfaces $\D_0$ and $\D_\infty$ are quadric surfaces, and so together they form an anticanonical divisor $\D = \D_0 + \D_\infty$, giving a Poisson structure $\pi$ on $\PP^3$.  If $\D_0$ and $\D_\infty$ are smooth and transverse, one can use the adjunction formula to show that the base locus
$\B = \D_0 \cap \D_\infty$ is a smooth curve of genus one.  Thus $\pi$ vanishes on an elliptic curve in $\PP^3$.  To find the remaining zeros of $\pi$, we must determine the singularities of the surfaces in the pencil.

To do so, we use the fact that a pair of homogeneous quadratic forms can always be put into a normal form.  More precisely, there exists homogeneous coordinates $x_0,\ldots,x_3$ and constants $a_0,\ldots,a_3 \in \CC$ such that
\begin{align*}
G &= x_0^2 + x_1^2+x_2^2 + x_3^2 & H &= a_0 x_0^2 + a_1 x_1^2 + a_2 x_2^2+a_3 x_3^2
\end{align*}
The surfaces in the pencil are then given by $
\tfrac{G}{H} = t$, or equivalently
\[
\lambda G - \mu H = \sum_{i=0}^3 (\lambda-\mu a_i) x_i^2 = 0 
\]
where $t =  \lambda^{-1}\mu$.  For generic values of $t$, the function $\lambda G - \mu H$ is a nondegenerate quadratic form, and hence its only critical point is the origin in $\CC^4$.  It follows that the corresponding surface $\D_t \subset \PP^3$ is smooth.

However, when $t = a_0^{-1}$, so that $\lambda = -\mu a_0$, the quadratic form has rank three, and this results in an isolated singularity of the surface $\D_t$ at the point $[1:0:0:0] \in \D_t \subset \PP^3$.  Similarly, the other values $t = a_i^{-1}$ give surfaces with isolated singularities at $[0:1:0:0]$, $[0:0:1:0]$ and $[0:0:0:1]$.  

We thus arrive at the following decomposition of $\PP^3$ into symplectic leaves:
\begin{itemize}
\item The four points $\bS = \{[1:0:0:0], \ldots, [0:0:0:1]\}$
are isolated symplectic leaves of dimension zero.  Near each of these points, one can find coordinates in which the Poisson structure takes the form of \autoref{ex:sl2}.
\item The base locus $\B = \D_0 \cap \D_\infty$ is an elliptic curve, and each of its points is a zero dimensional leaf.  Near such a point, the Poison structure is of  the type described in \autoref{ex:local-pencil}.
\item The two-dimensional leaves are given by the open sets $\D_t \setminus (\B \cup \bS)$ in the quadric surfaces $\D_t$ for $t \in \PP^1$.
\end{itemize}
Notice that $\Zeros(\pi)$ has components of dimensions zero and one.
\qed
\end{example}

\begin{example}\label{ex:r13}
Similarly, suppose that $G$ is an irreducible homogeneous cubic function and that $H$ is linear.  Then we obtain a pencil on $\PP^3$ from $f = H^{-3}G$.
Now we have $\D_0 = \Zeros(G)$ and $\D_\infty = 3 \Y$ where $\Y  \cong \PP^2 \subset \PP^3$ is the plane defined by $H = 0$.  Thus the polar divisor of $f^{-1}df$ is given by
\[
(\D_0 + \D_\infty)_{\textrm{red}} = \D_0 + \Y,
\]
which is evidently quartic.  Hence we obtain a Poisson structure whose leaves are determined by the pencil.

The base locus is the intersection $\D_0 \cap \Y$ which is a cubic curve.  If we assume that $G$ is sufficiently generic, this curve will be smooth, hence elliptic.  There are eight additional isolated points where the Poisson structure vanishes, corresponding to singularities of the surfaces in the pencil.
\qed
\end{example}

\subsection{Further constructions}
\label{sec:further-constr}
There are a number of other ways in which one can construct Poisson threefolds.  We leave the exploration of these constructions as exercises to the reader.

\paragraph{Closed meromorphic one-forms:}
Let $\X$ be a threefold, and let $\alpha$ be a closed one-form with poles on an anticanonical divisor $\D \subset \X$, cut out by the vanishing of a section $\mu \in \Gamma(\X,\acan_\X)$.  Then $\hook{\alpha}\mu$ is a holomorphic Poisson structure.  A special case is when $\alpha = d \log f$ for a pencil $f$ as above, but in general the symplectic leaves need not be the level sets of a meromorphic function.
  
\begin{exercise}\label{ex:log1111}
Let $[x_0:x_1:x_2:x_3]$ be homogeneous coordinates on $\X = \PP^3$, and let $\D = \Zeros(x_0x_1x_2x_3)$ be the union of the coordinate planes (an anticanonical divisor).  Show that every closed meromorphic one-form with first-order poles on $\D$ may be written in homogeneous coordinates as
\[
\alpha = \sum_{j=0}^3 \lambda_j \frac{dx_j}{x_j}
\]
where $\lambda_0,\ldots,\lambda_3 \in \CC$ are constants satisfying $\sum_{j=0}^3 \lambda_j =0$. Show that the induced symplectic foliation of $\PP^3 \setminus \D \cong (\CC^*)^3$ is regular, but for generic values of the constants $\lambda_0,\ldots,\lambda_3$, the leaves are not the level sets of any single-valued function. \qed
\end{exercise}

\paragraph{Horizontal lifts:} In \autoref{ex:higher-cohiggs}, we saw how to produce Poisson structures on $\PP^1$-bundles, where the base is equipped with the zero Poisson structure.  As another special case of the general construction in \cite[Section 6]{Polishchuk1997}, we can produce Poisson structures on $\PP^1$-bundles over surfaces with nontrivial Poisson structures as follows:

\begin{exercise} \label{ex:p1-over-surf}
Let $(\Y,\pi)$ be a Poisson surface with degeneracy divisor $\D \subset \Y$.  Suppose that $\cE$ is a rank-two vector bundle on $\X$ equipped with a meromorphic flat connection
\[
\nabla : \cE \to \cE \otimes \Omega^1_\X(\D).
\]
By this, we mean that $\nabla$ is a flat connection in the usual sense away from $\D$, but in a local trivialization near $\D$, it takes the form
\[
\nabla = d + f^{-1}A
\]
where $A$ is a holomorphic matrix-valued one-form, and $f$ is a local defining equation for $\D$.  Show that, even though the connection $\nabla$ is singular, the horizontal lift of $\pi$ to $\X = \PP(\cE)$ is holomorphic everywhere, so that $\X$ becomes a Poisson threefold.  Describe the symplectic leaves. \qed
\end{exercise}

\paragraph{Group actions:} Let $\G$ be a Lie group and let $\g$ be its Lie algebra.  We recall that a \defn{classical triangular $r$-matrix for $\G$} is a tensor $r \in \wedge^2 \g$ satisfying $[r,r]=0\in\wedge^3 \g$.  It corresponds to a left-invariant Poisson structure on $\G$.  

Now let $\X$ be a manifold equipped with an action of $\G$.  Then we may evidently push $r$ forward along the action map $\g \to \Gamma(\X,\tshf{\X})$ to obtain a Poisson structure on $\X$.

\begin{exercise}
Suppose that $\G$ is two-dimensional.  Then every $r \in \wedge^2 \g$ is a classical triangular $r$-matrix.  Describe the symplectic leaves of the induced Poisson structure on $\X$ in terms of the orbits of $\G$.  \qed
\end{exercise}

\begin{exercise}\label{ex:log1111-2}
Show that the Poisson structures on $\PP^3$ described in \autoref{ex:log1111} are induced by a classical triangular $r$-matrix for the group $\G = (\CC^*)^3$, where $\G$ acts on $\PP^3$ in the standard way, by rotation of the coordinates. \qed
\end{exercise}

\subsection{Poisson structures on $\PP^3$ and other Fano threefolds}

For some simple threefolds $\X$, the space $\Pois(\X) \subset \Gamma(\X,\wedge^2\tshf{\X})$  of Poisson structures can be described explicitly.  In particular, the space of Poisson structures on $\PP^3$ is quite well understood:

\begin{theorem}[\cite{Cerveau1996,Loray2013}]
The variety $\Pois(\PP^3)$ has six irreducible components, and there are explicit descriptions of the generic Poisson structures in each component.
\end{theorem}
We refer to \cite[Section 8]{Pym2013} and \cite{Pym2015} for a detailed description of the geometry of these Poisson structures and their quantizations.  Let us just remark that the Poisson structures in each component can be described by constructions that we have already seen: pencils, closed logarithmic one-forms, $r$-matrices, $\PP^1$-bundles, and blow-downs (now of threefolds instead of surfaces).  Moreover, the generic Poisson structure in each component vanishes on a curve and possibly also a finite collection of isolated points.  The six components can be distinguished by the structure of these curves.
 
In fact, the paper \cite{Cerveau1996} solves a slightly different problem: it gives the classification of certain codimension-one foliations on $\PP^n$ for $n\ge 3$ under the assumption that the singular set has no divisorial components; when $n=3$, these foliations coincide with the symplectic foliations of Poisson structures whose zero sets are unions of curves and isolated points.  The paper \cite{Loray2013} gives an alternative approach to the classification of such foliations when $n=3$, and completes the classification of Poisson structures by showing that the ones with no divisorial components in their zero sets are Zariski dense in $\Pois(\PP^3)$.  It also extends the result to a classification of  Poisson structures on rank-one Fano threefolds:
\begin{theorem}[\cite{Loray2013}]\label{thm:fano3}
Let $\X$ be a Fano threefold with second Betti number $b_2(\X) = 1$.  If $\Pois(\X) \ne \{0\}$, then $\X$ is one of the following manifolds
\begin{itemize}
\item The projective space $\PP^3$
\item A quadric or cubic hypersurface $\X \subset \PP^4$
\item A degree-six hypersurface $\X \subset \PP(1,1,1,2,3)$
\item A degree four hypersurface $\X \subset \PP(1,1,1,1,2)$
\item An intersection of two quadric hypersurfaces in $\PP^5$
\item The minimal $\SL{2,\CC}$-equivariant compactification of $\SL{2,\CC}/\G$, where $\G < \SL{2,\CC}$ is the binary octahedral or icosahedral subgroup.
\end{itemize}
In each case, there is an explicit list of irreducible components of $\Pois(\X)$, described in terms of the constructions mentioned above.
\end{theorem}

Fano threefolds have been completely classified; see \cite{Iskovskikh1999} for a summary.  To the author's knowledge at the time of writing, the classification of their Poisson structures is not yet complete.  It would be interesting to know whether there are any examples cannot be obtained from the constructions we have discussed.

%%%%%%%%%%%%%%%%%%%%%%%%%%%%%%%%%%%%%%%%%%%%%%%%%%%%%%%%%%%%%%%%%%%%%%%%
%%%%%%%%%%%%%%%%%%%%%%%%%%%%%%%%%%%%%%%%%%%%%%%%%%%%%%%%%%%%%%%%%%%%%%%%
%%%%%%%%%%%%%%%%%%%%%%%%%%%%%%%%%%%%%%%%%%%%%%%%%%%%%%%%%%%%%%%%%%%%%%%%

\section{Poisson subspaces and degeneracy loci}
\label{sec:degeneracy}
We now turn our attention to the geometry of higher-dimensional Poisson structures, with an emphasis on the singularities that arise.  As we have seen, the locus where a Poisson bracket vanishes is a key feature of Poisson surfaces and threefolds.  In dimension four and higher, it is possible for the Poisson structure to have leaves of many different dimensions, and we will want to understand how they all fit together.  This typically leads to complicated singularities, so it is helpful begin with a more systematic treatment of Poisson structures on singular spaces.

\subsection{Poisson subspaces and multiderivations}
\label{sec:subspaces}

Let us briefly recall the standard notion of an analytic subspace (or subscheme) of a complex manifold; see~\cite[Section 5.3]{Griffiths1994} for more details.  If $\X$ is a complex manifold, then a (closed) complex analytic subspace of $\X$ is a closed subspace $\Y \subset \X$ that is locally cut out by a finite collection of holomorphic equations.  More precisely, there is an ideal $\cI \subset \cO_\X$ that is locally finitely generated, such that $\Y$ is the simultaneous vanishing set of all elements of $\cI$.

The holomorphic functions on $\Y$ are given by $\cO_\Y = \cO_\X/\cI$.  Thus they depend on the ideal $\cI$, rather than just the set of points underlying $\Y$.  For example, if $x$ is the standard coordinate on $\X = \CC$, the ideals $(x^k) \subset \cO_\CC$  for different values of $k$ define different analytic subspaces with same underlying set $\Y = \{0\}$; their algebras of functions are given by $\cO_\Y \cong \CC \cdot 1 \oplus \CC \cdot x \oplus\cdots \oplus \CC\cdot x^{k-1}$ with $x^k=0$.

Now suppose that $(\X,\pi)$ is a complex Poisson manifold.  An analytic subspace $\Y \subset \X$ is a \defn{Poisson subspace} if it inherits a Poisson bracket from the one on $\X$ via the restriction map $\cO_\X \to \cO_\Y$.  From the definition, we immediately have the following
\begin{lemma}
For a closed analytic subspace $\Y \subset \X$ with ideal sheaf $\cI\subset \cO_\X$, the following are equivalent:
\begin{enumerate}
\item $\Y$ is a Poisson subspace
\item $\cI$ is a Poisson ideal, i.e.~$\{\cO_\X,\cI\}\subset \cI$
\item $\cI \subset \cO_\X$ is invariant under Hamiltonian flows
\end{enumerate}
\end{lemma}

\begin{example}
If $f \in \cO_\X$ is a Casimir function, then the level sets of $f$ are all Poisson subspaces.  Shifting $f$ by a constant, it is enough to check this for the zero level set, defined by the ideal $(f) = \cO_\X\cdot f \subset \cO_\X$.  Using $\{f,- \} = 0$, we compute
\[
\{\cO_\X,\cO_\X\cdot f\} = \{\cO_\X,\cO_\X\}f + \cO_\X\{\cO_\X,f\} = \{\cO_\X,\cO_\X\} f \subset (f),
\]
so that $(f)$ is a Poisson ideal, as required. \qed
\end{example}

\begin{exercise}\label{ex:intersects}
Show that if $\Y_1,\Y_2 \subset \X$ are Poisson subspaces with ideals $\cI_1$ and $\cI_2$, then $\cI_1 \cap \cI_2,\cI_1 \cdot \cI_2$ and $\cI_1+\cI_2$ are also Poisson ideals.  Geometrically, these operations correspond to the union, the union ``with multiplicities'' and the intersection, respectively.\qed
\end{exercise}

Notice that since a Poisson subspace $\Y\subset \X$ is invariant under all Hamiltonian flows, it is necessarily a union of symplectic leaves.  However, the converse need not hold; the main difficulty has to do with the possible existence of nilpotent elements in $\cO_\Y$.

Indeed, let us recall that if $\Y \subset \X$ is a closed analytic subspace in an arbitrary complex manifold, then the \defn{reduced  subspace $\Y_{\red} \subset \Y \subset \X$} is the unique analytic subspace of $\X$ that has the same underlying points as $\Y$, but has no nilpotent elements in its algebra of functions.  More precisely, if $\Y$ is defined by the ideal $\cI \subset \cO_\Y$, then $\Y_{\red}$ is defined by the \defn{radical ideal}
\[
\sqrt{\cI} = \set{f \in \cO_\X}{f^k \in \cI \textrm{ for some }k \in \ZZ_{>0}}
\]
and we have a natural inclusion of subspaces $\Y_{\red} \subset \Y$ corresponding to the reverse inclusion $\cI \subset \sqrt{\cI}$ of ideals.

\begin{exercise} Suppose that $Z \in \tshf{\X}$ is a vector field on $\X$ that is tangent to $\Y$, in the sense that its flow preserves the ideal $\cI$, i.e.~$Z(\cI) \subset \cI$.  Show that $Z$ is also tangent to $\Y_{\red}$.  Conclude that if  $\pi$ is a Poisson structure on $\X$ such that $\Y$ is a Poisson subspace, then $\Y_{\red}$ is also a Poisson subspace.\qed
\end{exercise}

\begin{exercise}\label{ex:red-poisson}
Equip $\X = \CC^3$ with the Poisson brackets
\begin{align}
\{x,y\} &= 2y  & \{y,z\} &= x  & \{z,x\} &= 2z \label{eqn:sl2}
\end{align}
from \autoref{ex:sl2}.  Find an analytic subspace $\Y \subset \X$ such that $\Y_{\red}$ is a Poisson subspace, but $\Y$ is not. (Hint: look for a zero-dimensional example.) \qed
\end{exercise}

\subsection{Vector fields and multiderivations}
The tangent spaces of an analytic space $\Y$ do not assemble into a vector bundle over the singular points, but we can still make sense of vector fields as derivations of functions:
\[
\tshf{\Y} = \mathrm{Der}(\cO_\Y).
\]
Clearly derivations of $\cO_\Y$ can be added together, and multiplied by elements of $\cO_\Y$, and hence they form an $\cO_\Y$-module.  This module structure is enough for many purposes in Poisson geometry; the lack of a tangent bundle is not much of an impediment.

\begin{exercise}\label{ex:sl2-der}
Recall from \autoref{ex:sl2} that the function $f = \tfrac{1}{2}x^2+2yz$ on $\CC^3$ generated the Poisson brackets \eqref{eqn:sl2}
via the Jacobian construction.  The zero set $\Y = f^{-1}(0) \subset \CC^3$  is an analytic subspace that has an isolated singular point at the origin.  Show that the Hamiltonian vector fields $H_x,H_y$ and $H_z$, together with the Euler vector field $x\cvf{x}+y\cvf{y}+z\cvf{z}$, induce derivations of $\cO_\Y$ that generate $\tshf{\Y}$ as an $\cO_\Y$-module. \qed
\end{exercise}

To formulate an analogue of the correspondence between Poisson brackets and bivectors, we recall that a \defn{multiderivation of degree $k$ on $\Y$} is a $\CC$-multilinear operator
\[
\underbrace{\cO_\Y \times \cdots \times \cO_\Y}_{k \textrm{ times}} \to \cO_\Y
\]
that is totally skew-symmetric and is a derivation in each argument.  We denote by
\[
\der[k]{\Y} = \{\textrm{multiderivations of degree }k \textrm{ on }\Y\}
\]
the sheaf of multiderivations; like $\tshf{\Y} = \der[1]{\Y}$, all these sheaves are $\cO_\Y$-modules.

  It is straightforward to define a Schouten-type bracket on multiderivations; see, e.g.~\cite[Chapter 3]{Laurent-Gengoux2013}.  In this way, we see that a Poisson bracket on $\Y$ is equivalent to a biderivation $\pi \in \der[2]{\Y}$ satisfying the integrability condition $[\pi,\pi] = 0 \in \der[3]{\Y}$.

From our experience with manifolds, it is tempting to think that a multiderivation should be the same thing as a section of  $\wedge^k\tshf{\Y}$, but this can fail when $\Y$ is singular.  More precisely, if we interpret $\wedge^k\tshf{\Y}$ as the $k$th exterior power of $\tshf{\Y}$ as an $\cO_\Y$-module, then there is a natural map $\wedge^k \tshf{\Y} \to \der[k]{\Y}$, defined by sending a wedge product $Z_1\wedge \cdots \wedge Z_k$ of vector fields to the multiderivation
\[
(f_1,\ldots,f_k) \mapsto \det(Z_i(f_j))_{1 \le i,j\le k}
\]
However, this map will typically fail to be an isomorphism at the singular points of $\Y$.  Poisson structures are always elements of $\der[2]{\Y}$, but they may not by defined by elements of $\wedge^2 \tshf{\Y}$.  Thus one should work with $\der[\bullet]{\Y}$ rather than $\wedge^\bullet \tshf{\Y}$ when doing Poisson geometry on singular spaces.

\begin{example}\label{ex:sl2-bider}
We continue \autoref{ex:sl2-der}.  Since $f$ is a Casimir function, $\Y$ is a Poisson subspace.  We claim that the corresponding biderivation $\pi \in \der[2]{\Y}$ is not in the image of the natural map $\wedge^2\tshf{\Y} \to \der[2]{\Y}$.

Indeed, by \autoref{ex:sl2-der}, $\tshf{\Y}$ is generated by derivations that vanish at the origin.  Therefore elements of $\wedge^2 \tshf{\Y}$ must vanish to order two there. But the Poisson bracket only vanishes to order one, so it cannot be given by an element of $\wedge^2\tshf{\Y}$.  More precisely, let $\fm = (x,y,z) \subset \cO_\Y$ denote the ideal of functions vanishing at the origin.  Then $Z(\cO_\Y) \subset \fm$ for any $Z \in \tshf{\Y}$.  Therefore, for any bivector $\eta \in \wedge^2\tshf{\Y}$, the corresponding multiderivation satisfies $\eta(\cO_\Y\times \cO_\Y) \subset \fm^2$.  But by definition of the Poisson bracket, we have $\pi(x,y) = 2y \notin \fm^2$.  \qed
\end{example}

\subsection{Degeneracy loci}

\subsubsection{Definition and basic properties}

We now turn to a natural class of Poisson subspaces that are key structural features of any Poisson manifold: the degeneracy loci.

Let $(\X,\pi)$ be a Poisson manifold, and let $k \ge 0$ be an integer.  The \defn{$2k$th degeneracy locus of $\X$} is the subset
\[
\Dgn{2k}(\pi) = \set{p \in \X}{\textrm{the symplectic leaf through }p\textrm{ has dimension}\le 2k}.
\]
It is a generalization of the zero locus $\Dgn{0}(\pi) = \Zeros(\pi)$.  Indeed, $\Dgn{2k}(\pi)$ is precisely the locus where the tensor $\pi$, viewed as a bilinear form on the cotangent spaces, has rank at most $2k$.  We therefore have the identification
\[
\Dgn{2k}(\pi) = \Zeros(\pi^{k+1}) \qquad \qquad \pi^{k+1} = \underbrace{\pi \wedge \cdots \wedge \pi}_{k\textrm{ times}} \in \wedge^{2k+2}\tshf{\X}
\]
In local coordinates, the locus $\Dgn{2k}(\pi)$ is the vanishing set of the Pfaffians of all $(2k+2)\times(2k+2)$ skew-symmetric submatrices of the matrix $(\{x_i,x_j\}_{i,j})$ of Poisson brackets.  In particular, it is an analytic subspace.  For a more invariant description of this ideal, we can observe that there is a natural map of $\cO_\X$-modules
\[
\xymatrix{
\forms[2k+2]_{\X} \ar[r]^-{\pi^{k+1}} & \cO_\X
}
\]
given by the pairing of polyvectors and forms.  The image of this map is a submodule in $\cO_\X$, hence an ideal; it is the ideal defining $\Dgn{2k}(\pi)$.

Thus the degeneracy loci give a filtration of $\X$ by closed analytic subspaces
\[
\Dgn{0}(\pi) \subset \Dgn{2}(\pi) \subset \Dgn{4}(\pi) \subset \cdots \subset \X.
\]
One can check that when $\Dgn{2k}(\pi) \ne \X$, the subspace $\Dgn{2k-2}(\pi) \subset \Dgn{2k}(\pi)$ is always contained in the singular locus of $\Dgn{2k}(\pi)$.  (This inclusion may or may not be an equality.)  Thus singularities are very common in the study of higher-dimensional Poisson brackets.

\begin{example}
Let $w,x,y,z$ be global coordinates on $\CC^4$, and consider the Poisson structure
\[
\pi = w\,\cvf{w}\wedge\cvf{x} + y\,\cvf{y}\wedge\cvf{z}
\]
The corresponding matrix of Poisson brackets is 
\[
\begin{pmatrix}
0 & w & 0 & 0 \\
-w & 0 & 0 & 0 \\
0 & 0 & 0 & y \\
0 & 0 & -y & 0
\end{pmatrix}
\]
Clearly $\pi$ vanishes on the plane $w=y=0$, so this plane is the degeneracy locus $\D_0(\pi)$.   Meanwhile $\D_2(\pi)$ is given by the vanishing of 
\[
\pi^2 = 2wy \, \cvf{w}\wedge\cvf{x}\wedge \cvf{y}\wedge\cvf{z}.
\]
Thus $\D_2(\pi)$ is the union of the hyperplanes $w=0$ and $y=0$, which is singular along their intersection.  So in this case, the inclusion $\Dgn{0}(\pi) \subset \Dgn{2}(\pi)_\sing$ is actually an equality. \qed
\end{example}

By definition, the locus $\D_{2k}(\pi)$ is a union of symplectic leaves, which suggests the following
\begin{proposition}[{\cite[Corollary 2.4]{Polishchuk1997}}]
The degeneracy locus $\Dgn{2k}(\pi)$ is always a Poisson subspace of $\X$.
\end{proposition} 

\begin{proof}
Let $\cI$ be the ideal defining $\D_{2k}(\pi)$ as above.  We need to show that $\cI$ is preserved by the flow of any Hamiltonian vector field $H_g$ for $g \in \cO_\X$.  But if $f \in \cI$, then $f$ is locally given by a pairing
\[
f = \abrac{ \pi^{k+1} , \omega }
\]
where $\omega \in \forms[2k+2]_\X$.  Therefore
\begin{align*}
\{g,f\} &= \lie_{H_g} f \\
&= \lie_{H_g}\abrac{ \pi^{k+1} , \omega } \\
&= \abrac{ \lie_{H_g}(\pi^{k+1}) , \omega } + \abrac{\pi^{k+1} ,  \lie_{H_g} \omega } \\
&= \abrac{ (k+1)\pi^k \wedge \lie_{H_g}\pi , \omega } + \abrac{\pi^{k+1} ,  \lie_{H_g} \omega } \\
&=  \abrac{\pi^{k+1} ,  \lie_{H_g} \omega } \\
&\in \cI
\end{align*}
where we have used the fact $\lie_{H_g} \pi = 0$.
\end{proof}

\subsubsection{Dimensions of degeneracy loci}

The study of degeneracy loci of vector bundle maps (also known as determinantal varieties) is a classical subject in algebraic geometry, and quite a lot is known about them. We recall some basic facts about degeneracy loci of skew forms; see, e.g. \cite{Harris1984a,Jozefiak1979}, for more details.

If $\cE$ is a vector bundle of rank $n$ and
\[
\rho \in \Gamma(\X,\wedge^2\cE)
\]
is a skew-symmetric form, then we may define its degeneracy loci $\Dgn{2k}(\rho)$ as we did for Poisson structures above.  In this setting, there is a bound on the codimension
\begin{align}
\codim\, \Dgn{2k}(\rho) \le {n -2k \choose 2}, \label{eqn:codim-bound}
\end{align}
which can be understood intuitively as follows.  Suppose $p \in \X$ is a point where the rank of $\rho$ is equal to $2k$.  Then near $p$ we can choose collections of sections $e_1,\ldots,e_k$, $f_1,\ldots,f_k$ and $g_1,\ldots,g_{n-2k}$ that give a basis for $\cE$ in which $\rho$ takes the form
\[
\rho = \sum_{j=1}^k e_j \wedge f_j + \sum_{1 \le i < j \le n-2k} h_{ij}g_i\wedge g_j
\]
where $h_{ij}$ are functions that vanish at $p$.  It is clear that the rank of $\rho$ drops to $2k$ precisely where the functions $h_{ij}$ vanish.  Since there are $n-2k$ of these we get the expected bound on the codimension.  With more effort, one can upgrade this argument to work at points $p \in \D_{2k}(\rho)$ where $\rank\, \rho < 2k$.

When the bound \eqref{eqn:codim-bound} is an equality one has relatively good control over the singularities of $\Dgn{2k}(\rho)$ that can occur, as well as formulae for their fundamental classes in the cohomology of $\X$. But when the codimension is smaller than the bound, we are in the situation of \defn{excess intersection}, and it is rather more difficult to control what is happening.

One of the interesting features of Poisson structures is that excess intersection is quite common.  To see why this must be true, notice that if $p$ is a point where the rank of $\pi$ is equal to $2k$, then the entire symplectic leaf through $p$ clearly lies in $\Dgn{2k}(\pi)$.  Hence the codimension of $\Dgn{2k}(\pi)$ in a neighbourhood of $p$ at most $n-2k$, which is rather less than \eqref{eqn:codim-bound} would suggest.  On the other hand, there are also examples where
\[
\codim \, \D_{2k}(\pi) > n-2k.
\]
In these situations there are no leaves of dimension $2k$, but there are leaves of dimension less than $2k$.  Here is an example when $k=1$:

\begin{exercise}
Let $\pi$ be the linear Poisson structure associated with the Lie algebra $\sln{3,\CC}$.  Show that $\Dgn{2}(\pi) = \Dgn{0}(\pi) = \{0\}$ as sets.  Note that $\Dgn{2}(\pi)$ is not reduced in this case, so the ideals are different. \qed
\end{exercise}

There is an intriguing conjecture concerning the excess dimension of degeneracy loci on compact Poisson varieties.
\begin{conjecture}[Bondal~\cite{Bondal1993}]
Suppose that $\X$ is a Fano manifold and $\pi$ is a Poisson structure on $\X$.  Then for every $k$ such that $0 \le 2k \le \dim \X$, the degeneracy locus $\D_{2k}(\pi)$ has an irreducible component of dimension $\ge 2k+1$.
\end{conjecture}
Beauville~\cite{Beauville2011} has suggested that the Fano condition could likely be considerably relaxed.  

The author is unaware of any counterexamples to the conjecture.  Indeed, it has been verified in some special cases.  The first main result is the following
\begin{theorem}[{\cite[Proposition 4]{Beauville2011},\cite[Corollary 9.2]{Polishchuk1997}}]
Let $(\X,\pi)$ be a compact Poisson manifold.  Suppose that $\rank(\pi) = 2k+2$ on an open dense set, and that
\[
c_1(\X)^{n-2k-1} \ne 0.
\]
Then $\Dgn{2k}(\pi)$ has a component of dimension at least $2k+1$. 
\end{theorem}

In other words, the conjecture always holds for the rank drop locus of maximal dimension.  This implies, in particular, that Bondal's conjecture holds for Fano threefolds.  (The three-dimensional case could also be deduced as a consequence of Druel's classification~\autoref{thm:druel}, but that result is much harder to prove.) The proof is a straightforward application of Bott's theorem~\cite{Bott1972}, which gives a topological obstruction to the existence of a regular foliation on a manifold, in terms of Chern classes.

The second main result implies that the conjecture holds for Fano fourfolds:
\begin{theorem}[\cite{Gualtieri2013a}]\label{thm:gp}
Let $(\X,\pi)$ be a Fano manifold of dimension $n=2d$.  Then $\Dgn{2d-2}(\pi)$ has a component of dimension at least $n-1$ and $\Dgn{2d-4}(\pi)$ has a component of dimension at least $2d-3$.
\end{theorem}
In this case Bott's vanishing theorem cannot be directly applied.  The proof is based on some new tools that relate the Poisson geometry of the canonical bundle to the singularities of the subspace $\Dgn{2d-2}(\pi) \subset \X$.

Another piece of evidence for the conjecture comes from the fact that the adjunction procedure described in \autoref{sec:adjunction} has a natural generalization to the higher degeneracy loci~\cite{Gualtieri2013a}.   Taking the derivative of $\pi^{k+1}$ along its zero locus, and applying the contraction $\cotshf{\X}\otimes\wedge^{2k+2}\tshf{\X} \to \wedge^{2k+1}\tshf{\X}$, one produces a \defn{modular residue} that is a multiderivation of degree $2k+1$ on $\Dgn{2k}(\pi)$.  This residue is nonzero in many examples, which hints at a possible deeper geometric explanation for the dimensions appearing in the conjecture.

\section{Log symplectic structures}
\label{sec:log-symp}
As the dimension of $\X$ increases, the singularities of the degeneracy loci typically become much more complicated.  To get some control over the situation, it is helpful to make some simplifying assumptions.  In the past several years, starting with the works \cite{Goto2002,Guillemin2014,Radko2002}, considerable attention has been devoted to the case of $C^\infty$ and holomorphic Poisson structures that have an open dense symplectic leaf, and degenerate in a simple way along a hypersurface.  In this final section, we give an introduction to these ``log symplectic'' manifolds, and describe some of the classification results that have been obtained. 

\subsection{Definition and examples}
Suppose that $(\X,\pi)$ is a Poisson manifold that is generically symplectic, i.e.~it has an open dense symplectic leaf. Then the dimension $n = \dim \X$ is necessarily even, and the maximal degeneracy locus $\D = \Dgn{n-2}(\pi) \subset \X$ is described by the vanishing of the top power
\[
\pi^{n/2} \in \Gamma(\X,\acan_\X),
\]
i.e.~it is an anticanonical divisor in $\X$ when it is nonempty.

\begin{definition}
A generically symplectic Poisson manifold $(\X,\pi)$ is \defn{log symplectic} if the anticanonical divisor $\D$ is reduced, i.e.~all components have multiplicity one. 
\end{definition}
When we want to emphasize the degeneracy divisor, we will say that the triple $(\X,\D,\pi)$ is a log symplectic manifold.
 
\begin{exercise}\label{ex:poisson-cpts}
Let $(\X,\pi)$ be a log symplectic manifold.  Show that every irreducible component of the degeneracy divisor $\D$ is a Poisson subspace of $\X$.  \qed
\end{exercise}

\begin{remark}
In fact, a theorem of Seidenberg~\cite{Seidenberg1967} implies that the irreducible components of any analytic space or scheme $\Y$ are automatically preserved by any vector field on $\Y$.  This implies that the irreducible components of any Poisson space are always Poisson subspaces. \qed
\end{remark}

The reason for the terminology ``log symplectic'' is as follows.  Consider the subsheaf
\[
\tshf{\X}(-\log \D) \subset \tshf{\X}
\]
consisting of all vector fields tangent to $\D$.  Near a smooth point $p \in \D$, we can choose local coordinates $x_1,\ldots,x_{n}$ on $\X$ such that $\D = \{x_1 = 0\}$.  We then have a basis
\[
\tshf{\X}(-\log \D) \cong  \abrac{ x_1\cvf{x_1}, \cvf{x_2},\ldots,\cvf{x_{n}}}
\]
for $\tshf{\X}(\log \D)$ as an $\cO_\X$-module.  In other words, away from the singular locus, $\tshf{\X}(-\log \D)$ is the sheaf of sections of a vector bundle.  (In general, it may fail to be a vector bundle over the singular locus, but it is the next best thing: a reflexive sheaf, meaning that it is isomorphic to its double dual.)

The dual of $\tshf{\X}(-\log \D)$ is the sheaf $\forms[1]_\X(\log \D)$ of differential one-forms with logarithmic singularities along $\D$ in the sense of \cite{Deligne1970,Saito1980}.  In coordinates near a smooth point as above, we have a basis
\[
\forms[1]_\X(\log \D) \cong \abrac{ \tfrac{dx_1}{x_1}, dx_2,\ldots,dx_{n}}
\]
and the term ``logarithmic'' comes from the fact that $\tfrac{dx_1}{x_1} = d \log x_1$.

One can show that the Poisson structure $\pi$ is log symplectic if and only if the isomorphism $\tshf{\X\setminus\D} \cong \forms[1]_{\X\setminus\D}$ defined by the symplectic structure on $\X\setminus\D$ extends to an isomorphism
\[
 \tshf{\X}(-\log \D) \cong \forms[1]_{\X}(\log \D)
\]
of $\cO_\X$-modules, even over the singular points of $\D$.  What this means is that we can invert $\pi$ to obtain a logarithmic two-form
\[
\omega \in \forms[2]_\X(\log \D)
\]
which is called the \defn{log symplectic form}.  Near a smooth point of $\D$, one can find log Darboux coordinates $(p_i,q_i)_{i=1,\ldots,n/2}$ so that
\begin{align}
\omega = \frac{dp_1}{p_1}\wedge dq_1 + dp_2\wedge dq_2 + \cdots + dp_{n/2}\wedge dq_{n/2}; \label{eqn:log-Darboux}
\end{align}
see \cite{Goto2002,Guillemin2014}.  Near a singular point, though, the structure can be much more complicated.

We recall now several examples of log symplectic manifolds, some of which are taken from \cite{Goto2002}.
\begin{example}[Surfaces]
For $\X = \PP^2$, most Poisson structures are log symplectic; the only cases that are not log symplectic are cases (h) and (i) in \autoref{fig:cubic}, which have components with multiplicity greater than one.  In contrast, the Poisson ruled surfaces obtained from compactified cotangent bundles (\autoref{sec:cotan-compact}) are never log symplectic, since the section at infinity has multiplicity two.\qed
\end{example}

\begin{example}[Products]
If $(\X,\D,\pi)$ and $(\X',\D',\pi')$ are log symplectic, then so is the product $(\X\times \X', \D\times\X'+\X\times\D',\pi+\pi')$. \qed
\end{example}

\begin{example}[Hilbert schemes]
If $(\X,\pi)$  is a log symplectic surface, then by the previous example $\X^n$ carries a log symplectic structure that is invariant under permutations of the factors.  It therefore descends to a Poisson structure on the singular variety $\X^n/\bS_n$, and this Poisson structure lifts to a log symplectic structure on the natural resolution of singularities---the Hilbert scheme $\Hilb^n(\X)$.  See \cite{Ran2016,Ran2017} for a discussion of these Poisson structures and a proof that they have unobstructed deformations.  We remark that a similar construction works when $(\X,\pi)$ is replaced with certain \emph{noncommutative} surfaces; see \cite{Nevins2007,Rains2016}. \qed
\end{example}

\begin{example}[Monopoles]
The reduced moduli space of $\SU{2}$-monopoles of charge $k$ is a $2k$-dimensional holomorphic symplectic manifold.  It has a natural compactification to a log symplectic structure on $\PP^{k-1}\times \PP^{k-1}$. \qed
\end{example}

\begin{example}[Bundles on elliptic curves]\label{ex:FO}
In \cite{Feigin1989,Feigin1998}, Feigin and Odesskii introduced families of Poisson structures $q_{n,r}(\Y)$ on $\PP^{n-1}$, determined by an elliptic curve $\Y$ and coprime integers $(n,r)$.   These structures are obtained by viewing $\PP^{n-1}$ as a certain moduli space of bundles over $\Y$, and they exhibit some beautiful classical geometry related to the embedding of $\Y$ as an elliptic normal curve in $\PP^{n-1}$. For example, the Poisson structure $q_{2d+1,1}$ on $\PP^{2d}$ is log symplectic, and for any $k \ge 0$, the degeneracy locus  $\Dgn{2k}$ is the union of all of the $k$-planes in $\PP^{n-1}$ that meet $\Y$ in at least $k+1$ points (the secant $k$-planes).\qed 
\end{example}

\begin{example}[Lie algebras]
If $\g$ is a Lie algebra, then $\g^\vee$ is a Poisson manifold, so it makes sense to ask whether $\g^\vee$ is log symplectic.  There are several Lie algebras for which this is the case; examples include:
\begin{itemize}
\item the Lie algebra $\aff{n} = \gln{n,\CC}\ltimes \CC^n$ of affine motions of $\CC^n$
\item the Borel subalgebras in the symplectic Lie algebras $\spn{2n,\CC}$ for $n \ge 1$
\item the logarithmic cotangent bundles of linear free divisors~\cite{Buchweitz2006,Granger2009}.
\end{itemize}
For any Lie algebra $\g$, the symplectic leaves in $\g^\vee$ are the coadjoint orbits.  Thus a necessary condition for $\g^\vee$ to be log symplectic is that it has an open coadjoint orbit (so that it is a so-called \defn{Frobenius Lie algebra}).  But this condition is not sufficient, since we also need the degeneracy divisor to be reduced.  Frobenius Lie algebras up to dimension six have been classified~\cite{Csikos2007}, and consulting the classification we see that most of them do not have log symplectic duals.  Those that do fall into the list of examples above. \qed
\end{example}

In almost all of these examples of log symplectic manifolds, the degeneracy divisor is highly singular.  Indeed, it seems that log symplectic manifolds with smooth degeneracy hypersurfaces are quite rare.  For example, this is impossible for Fano manifolds by \autoref{thm:gp}.  There is also a simple topological obstruction:
\begin{theorem}[\cite{Gualtieri2013a}]\label{thm:dsing}
Let $(\X,\D,\pi)$ be a log symplectic manifold.  If $\D$ is singular, then its singular locus has codimension at most three in $\X$.  When the codimension is equal to three, the singular locus is Gorenstein, and its fundamental class is
\begin{align*}
[\D_{\sing}] = c_1c_2-c_3 \in \cohlgy[6]{\X,\ZZ}
\end{align*}
\end{theorem}

\subsection{Some classification results}
One reason for focusing on log symplectic structures is that they play an important role in the classification of Poisson brackets on compact complex manifolds.  The point is that small deformations of a log symplectic structure remain log symplectic. Hence for a given compact complex manifold $\X$, the log symplectic structures on $\X$ form a Zariski open set $\LogSymp(\X) \subset \Pois(\X)$.  As a result, the closure of $\LogSymp(\X)$ is a union of irreducible components of $\Pois(\X)$, and we may try to understand these components as a first step towards understanding the full classification of Poisson structures on $\X$.

This direction of research is quite recent, but there are already some nontrivial results.  The typical strategy is to make some natural constraints on the singularities of the hypersurface $\D$, and classify the log symplectic structures that have the given singularity type.  

The simplest case is when $\D$ is the union of smooth components meeting transversely (simple normal crossings).   Such pairs $(\X,\D)$ are the output one obtains when one attempts to resolve the singularities of an arbitrary divisor.  In other words, the singularities of a simple normal crossing divisor cannot be resolved any further; they are, in some sense, the minimally singular examples.  Since the only singularities come from intersections of the components, this case has a combinatorial flavour and one can obtain some fairly good control over the geometry.

Even-dimensional toric varieties give many examples of simple normal crossings log symplectic manifolds, via the $r$-matrix construction for group actions mentioned in \autoref{sec:further-constr}.   The main classification result, due to R.~Lima and J.~V.~Pereira, states that beyond dimension two, these are the only examples amongst Fano manifolds with cyclic Picard group:
\begin{theorem}[\cite{Lima2014}] \label{thm:onlyPn}
Let $(\X,\D,\pi)$ be a log symplectic Fano manifold of dimension $n \ge 4$, with second Betti number $b_2(\X) = 1$.  Suppose that $\D$ is a simple normal crossings divisor.  Then $\X \cong \PP^n$, the divisor $\D$ is the union of the coordinate hyperplanes, and the Poisson structure is induced by an $r$-matrix for the standard torus action.  It therefore takes the following form in standard affine coordinates $x_1,\ldots,x_n$:
\[
\pi = \sum_{i<j} \lambda_{ij} (x_i \cvf{x_i}) \wedge (x_j\cvf{x_j})
\]
where $(\lambda_{ij})_{i,j} \in \CC$ is a nondegenerate skew-symmetric matrix.
\end{theorem} 
The original proof of this result consisted of an inductive argument that reduced the problem to the  classification of Poisson structures on Fano threefolds (\autoref{thm:fano3}) by repeatedly intersecting the irreducible components of $\D$.  In the final section of these notes, we give a new proof, which is more self-contained.  It is based on a similar inductive process, but the steps are simplified using some natural constraints on the characteristic classes of $(\X,\D)$.

The theorem suggests that the simple normal crossings condition is quite restrictive.  Indeed, most of the natural examples described in the previous section are not of this type, although one can sometimes (but not always) resolve the singularities to get to the simple normal crossings case; see, e.g.~\cite{Ran2017}. To make further progress on the classification, we must confront the singularities of the degeneracy divisor.  The author's paper \cite{Pym2016} developed some cohomological techniques for constraining the singularities of $\D$ and producing normal forms for the Poisson brackets near a singular point.  It explained that when the singular locus has codimension three (the maximum possible by \autoref{thm:dsing}), the local geometry is generically governed by an elliptic curve.  The main classification result on such ``elliptic'' structures is the following.
\begin{theorem}[\cite{Pym2016}]\label{thm:elliptic}
Suppose that $\pi$ is a log symplectic structure on $\PP^4$ whose degeneracy divisor $\D$ has the following properties:
\begin{itemize}
\item  the singular locus of $\D$ has codimension three in $\PP^4$; and
\item the modular residue is nonvanishing on the zero locus of $\pi$.
\end{itemize}
Then $\D$ is the secant variety of an elliptic normal curve $\Y \subset \PP^4$, and $\pi$ belongs to the family $q_{5,1}(\Y)$ described by Feigin and Odesskii (\autoref{ex:FO}). 
\end{theorem}
The paper also rules out the existence of such log symplectic structures on some other simple Fano fourfolds.  It is not clear how to extend the approach to higher dimensions: in the elliptic case, $\D$ has only one irreducible component, so the inductive approach used for \autoref{thm:onlyPn} does not work here. 

The Poisson structures described in \autoref{thm:onlyPn} and \autoref{thm:elliptic} sweep out Zariski open sets in the space $\Pois(\PP^n)$ of Poisson structures, and hence they give irreducible components.  Some other examples of log symplectic structures on projective space are known~\cite[Section 7.5]{Pym2013}, but the full classification is, at present, unresolved.

\subsection{The simple normal crossings case}

We now describe some basic structural facts about log symplectic manifolds with simple normal crossings degeneracy divisors, leading to a proof of \autoref{thm:onlyPn}.

\subsubsection{Topological constraints from residue theory}

Suppose that $\D \subset \X$ is a simple normal crossings divisor.
Near a point $p \in \D$ where exactly $k$ components meet, the transversality of the components implies that we may find coordinates $x_1,\ldots,x_n$ such that $\D$ is the vanishing locus of the product $x_1\cdots x_k$.  In these coordinates, $\D$ is simply the union of the coordinate hyperplanes $\D_j = \{x_j=0\}$ for $1 \le j \le k$.  Moreover, there are natural $\cO_\X$-module bases
\begin{align}
\begin{aligned}
\tshf{\X}(-\log \D) &\cong \abrac{x_1\cvf{x_1},\ldots,x_k\cvf{x_k},\cvf{x_{k+1}},\ldots,\cvf{x_n}}   \\
\forms[1]_\X(\log \D) &\cong \abrac{ \frac{dx_1}{x_1},\ldots,\frac{dx_k}{x_k},dx_{k+1},\ldots,dx_n}. 
\end{aligned}\label{eqn:tlog-basis}
\end{align}
So in this case, these sheaves are vector bundles, even at the singular points, and a log symplectic form is an element of $\forms[2]_\X(\log\D) = \wedge^2 \forms[1]_\X(\log \D)$.

The existence of a log symplectic structure puts strong constraints on the topology of the pair $(\X,\D)$.  To see this, we shall use a bit of residue theory.  We recall from \cite{Deligne1970,Saito1980} that for each component $\D_j$, there is a natural $\cO_\X$-linear residue map
\[
\forms[1]_\X(\log \D) \to \cO_{\D_j}
\]
which sends $\frac{dx_j}{x_j}$ to the function $1 \in \cO_{\D_j}$, and sends all other basis elements to zero.
\begin{exercise}
Verify that the residue map is independent of the chosen coordinates. \qed
\end{exercise}

Notice that the holomorphic forms $\forms[1]_\X$ are naturally a submodule in $\forms[1]_{\X}(\log \D)$; they are the elements whose residue on every irreducible component vanishes.  In this way we obtain the residue exact sequence
\begin{align}
\xymatrix{
0 \ar[r] & \forms[1]_\X \ar[r] & \forms[1]_\X(\log \D) \ar[r] &\bigoplus_{j=1}^k \cO_{\D_j} \ar[r] & 0. 
}\label{eqn:residue-seq}
\end{align}
On the other hand, if we denote by $\cO_\X(-\D_j) \subset \cO_\X$ the ideal of functions vanishing on $\D_j$, then we have the exact sequence
\begin{align}
\xymatrix{
0 \ar[r] & \cO_\X(-\D_j) \ar[r] & \cO_\X \ar[r] & \cO_{\D_j} \ar[r] & 0
}\label{eqn:ideal-seq}
\end{align}
Since $\D_j$ is locally defined by a single equation, $\cO_\X(-\D_j)$ is a locally free $\cO_\X$-module of rank one (a line bundle).  We denote by
\[
[\D_j] = -c_1(\cO_\X(-\D_j)) \in \cohlgy[2]{\X,\ZZ}
\]
the first Chern class of its dual.  When $\X$ is compact, this class is the Poincar\'e dual of the fundamental class of $\X$.  We can now prove the following statement.
\begin{proposition}\label{prop:chern-sinh}
Suppose that the simple normal crossings divisor $\D\subset \X$ is the degeneracy divisor of a log symplectic Poisson structure.  Then the following identity holds in the cohomology ring of $\X$:
\begin{align}
\ch(\tshf{\X})-\ch(\cotshf{\X}) = 2\sum_{j=1}^k \sinh\, [\D_j], \label{eqn:chern-sinh}
\end{align}
where $\ch(\tshf{\X})$ and $\ch(\cotshf{\X})$ are the Chern characters of the tangent and cotangent bundles.
\end{proposition}

\begin{proof}
By definition of the Chern character, we have $\ch(\cO(-\D_j)) = e^{-[\D_j]}$.  Using the additivity of the Chern character on the exact sequences \eqref{eqn:residue-seq} and \eqref{eqn:ideal-seq}, we obtain the identity
\begin{align}
\ch(\cotshf{\X}) = \ch(\forms[1]_\X(\log \D)) + \sum_{j=1}^k(e^{-[\D_j]}-1) \label{eqn:chern1}
\end{align}
Now the Chern character of a dual bundle is obtained by switching the signs of the odd components $\ch_1,\ch_3,\ch_5,\ldots$.  Therefore
\begin{align}
\ch(\tshf{\X}) = \ch(\tshf{\X}(-\log \D)) + \sum_{j=1}^k(e^{[\D_j]}-1). \label{eqn:chern2}
\end{align}
and the result follows by subtracting \eqref{eqn:chern1} from \eqref{eqn:chern2}.
\end{proof}

\begin{remark}
Formula \eqref{eqn:chern-sinh} is an equation in the even cohomology $\cohlgy[2\bullet]{\X,\ZZ}$, but the projections of both sides to $\cohlgy[4\bullet]{\X,\ZZ}$ are zero.  Hence the nontrivial conditions lie in the summands $\cohlgy[2j]{\X,\ZZ}$ where $j$ is odd, i.e.~they concern the components $\ch_1,\ch_3,\ch_5,\ldots$ of the Chern character. \qed
\end{remark}

This result is already sufficient to classify simple normal crossings log symplectic structures on projective space:
\begin{proposition}\label{prop:toricPn}
Let $\pi$ be a log symplectic structure on $\PP^n$ for $n > 2$, with simple normal crossings degeneracy divisor $\D$.  Then $\D$ is projectively equivalent to the union of the coordinate hyperplanes, and the Poisson structure is induced by an $r$-matrix as in \autoref{thm:onlyPn}.
\end{proposition}

\begin{proof}
We recall that the cohomology ring is given by
\[
\cohlgy{\PP^n,\ZZ} \cong \ZZ[H]/(H^{n+1}),
\]
where $H \in \cohlgy[2]{\PP^n,\ZZ}$ is the fundamental class of a hyperplane, and the Chern character is given by $\ch(\tshf{\PP^n}) = (n+1)e^H-1$.

Let $d_1,\ldots,d_k$ be the degrees of the irreducible components of $\D$, so that the fundamental classes are given by $[\D_j] = d_jH$.  Thus \eqref{eqn:chern-sinh} reads
\[
(n+1)e^H - (n+1)e^{-H} = \sum_{j=1}^k e^{d_j H} - \sum_{j=1}^k e^{-d_j H}.
\]
Considering the coefficients of odd powers of $H$, we obtain the equations
\begin{align*}
n+1 &=d_1 + \cdots + d_k \\
n+1 &=d_1^3 + \cdots + d_k^3  \\
    &\ \ \vdots \\
n+1 &=d_1^{n-1} + \cdots + d_k^{n-1}
\end{align*}
for the positive integers $d_1,\ldots,d_k$.  Since $n > 2$, there is more than one nontrivial equation in this list, and evidently the only simultaneous solution is given by $k = n+1$ and $d_1=d_2=\cdots=d_k = 1$.  In this way, we see that $\D$ must be the union of $n+1$ hyperplanes with normal crossings.  Such an arrangement is projectively equivalent to the coordinate hyperplanes.  

Now consider the standard action of the $n$-dimensional torus on $\PP^n$ by rescaling the coordinates.  It preserves the divisor $\D$.  In standard affine coordinates $x_1,\ldots,x_n$ the infinitesimal generators are the vector fields $x_1\cvf{x_1},\ldots,x_n\cvf{x_n}$.  From \eqref{eqn:tlog-basis}, we see that they form a global basis for $\tshf{\PP^n}(-\log \D)$.  Therefore $\tshf{\PP^n}(-\log \D)$ is a trivial vector bundle over $\PP^n$ whose fibres are canonically identified with the Lie algebra $\ft$ of the torus.  Since holomorphic sections of a trivial bundle on a compact manifold are always constant, we obtain an isomorphism of vector spaces $\Gamma(\PP^n,\wedge^2\tshf{\PP^n}(-\log \D)) \cong \wedge^2 \ft$, 
which implies that the Poisson structure comes from an $r$-matrix and has the desired form in coordinates.
\end{proof}

\begin{exercise}
Let $\X = \PP^{n_1}\times \cdots \times \PP^{n_k}$ be a product of projective spaces of dimensions $n_i > 2$, and let $\pi$ be a log symplectic structure on $\X$ with simple normal crossings degeneracy divisor.  By combining the method in the proof of \autoref{prop:chern-sinh} with K\"unneth's formula for the cohomology of a product, show that $\pi$ is induced by an $r$-matrix for the standard action of the torus $(\CC^*)^{n_1}\times \cdots \times (\CC^*)^{n_k}$ on $\X$.  \qed
\end{exercise}

\subsubsection{Intersection of components and biresidues}

Let $\D \subset \X$ be a simple normal crossings divisor.  Suppose given $k$ components $\D_1,\ldots,\D_k$ of $\D$.  Let $\W = \D_1 \cap \cdots \cap \D_k$ be their intersection, and let $\Wc \subset \W$ be the open dense subset obtained by removing the intersections with any other components of $\D$.

Notice that $\cO_{\D_j}|_\W =\cO_\W$ is the trivial bundle.  Since the residue maps on each component are $\cO_\X$-linear, they give rise to a map
\[
\forms[1]_{\X}(\log \D)|_\W \to \bigoplus_{j=1}^k \cO_{\W}
\]
of vector bundles on $\W$.  From the expressions \eqref{eqn:tlog-basis} for the generators of $\forms[1]_{\X}(\log \D)$ it is apparent that this map give rise to a canonical exact sequence
\[
\xymatrix{
0 \ar[r] &\forms[1]_{\Wc} \ar[r] & \forms[1]_\X(\log\D)|_\Wc \ar[r] & \bigoplus_{j=1}^k \cO_\Wc \ar[r] & 0.
}
\]
over the open subset $\Wc\subset \W$.

Taking exterior powers, we get an induced map
\[
\Bires_\W : \forms[2]_\X(\log \D) \to {\bigwedge}^2 \bigoplus_{j=1}^k \cO_\W
\]
called the \defn{biresidue}.  (More generally, there is a $j$-residue, defined for forms of degree $j \le k$.)  In coordinates, an element $\omega \in \forms[2]_\X(\log\D)$ may be written in the form
\[
\sum_{1 \le i<j \le k} \lambda_{ij} \frac{dx_i}{x_i}\wedge \frac{dx_j}{x_j} + \cdots
\]
where $\lambda_{ij}$ are holomorphic functions and $\cdots$ denotes terms involving at most one pole.  Then the biresidue simply extracts the skew matrix $\lambda_{ij}|_\W$.

The biresidue can alternatively be described by applying the residue map twice.  Given a component $\D_i$, the residue of $\omega$ along $\D_i$ is a one-form $\Res_{\D_i} \omega$ which has logarithmic poles along the intersection $\D_i \cap (\D\setminus \D_i)$.  Hence given another component $\D_j$, we can compute
\[
\Res_{\D_j}\Res_{\D_i}\omega \in \cO_{\D_i\cap \D_j}
\]
and the restriction of this function to $\W \subset \D_i\cap \D_j$ gives the component $\lambda_{ij}|_\W$ of the biresidue $\Bires_\W(\omega)$.

Now let us incorporate a log symplectic Poisson structure $\pi$ on $(\X,\D)$ and take $\omega = \pi^{-1}$ to be the log symplectic form.  We note that  \autoref{ex:intersects} and \autoref{ex:poisson-cpts} together imply that the intersection $\W = \D_1 \cap \cdots  \cap\D_k$ is a Poisson subspace.  Indeed, we have the following commutative diagram
\[
\xymatrix{
0 \ar[r] &\forms[1]_{\W} \ar[r]\ar[d]_{\pi|_\W} & \forms[1]_\X(\log\D)|_\W \ar[r] \ar@/_/[d]_{\pi} & \bigoplus_{j=1}^k \cO_\W \ar[r] & 0 \\
0  & \tshf{\W} \ar[l] & \tshf{\X}(-\log \D)|_\W \ar[l] \ar@/_/[u]_{\omega}&\bigoplus_{j=1}^k \cO_\W\ar[l] \ar[u]_{\Bires_\W (\omega)} & 0  \ar[l] 
}
\]
The rows of this diagram are exact over the open set $\Wc \subset \W$, and hence we deduce the following
\begin{lemma}\label{lem:bires}
The Poisson structure $\pi|_\W$ is symplectic at a point $p \in \Wc$ if and only the pairing on  $\bigoplus_{j=1}^k \cO_{\W}$ defined by $\Bires_\W(\omega)$ is nondegenerate at $p$.
\end{lemma}

With these basic facts in hand, we now turn to the proof of \autoref{thm:onlyPn}.  The first step is to show that the divisor must have many components:
\begin{proposition}\label{prop:many-cpts}
Let $(\X,\pi)$ be a log symplectic Fano manifold with $b_2(\X) = 1$ and simple normal crossings degeneracy divisor $\D$.  Then $\D$ has at least $\dim \X - 1$ components.
\end{proposition}

\begin{proof}
This is a simplified version of the argument from \cite{Lima2014}.  The case $\dim \X = 2$ is obvious: the Fano condition ensures that the anticanonical line bundle is nontrivial, and hence the degeneracy divisor is nonempty, which is all that is needed.  Hence we may assume that $\dim \X \ge 4$.  

We will make repeated use of the fact that every intersection of the form $\D_{i_1}\cap \cdots \cap \D_{i_k}$ with $k < n$ is nonempty and connected.  Since $b_2(\X) = 1$, every hypersurface in $\X$ is ample, so this statement about intersections is a consequence of the Lefschetz hyperplane theorem~\cite[Section 1.2]{Griffiths1994}.

We begin by proving that the divisor must have at least three components which all intersect.   Indeed, because $\X$ is Fano, $\D$ is certainly nonempty, so it has at least one component, say $\D_1$.  Now consider the residue $\alpha = \Res_{\D_1} \omega$, which is a one-form on $\D_1$ that may have poles along the intersections of $\D_1$ with the other components.  By the local normal form \eqref{eqn:log-Darboux}, this  one-form is evidently nonzero.  But since $\D_1$ is an ample divisor, the Lefschetz hyperplane theorem implies that the restriction map  $\Gamma(\X,\forms[1]_\X) \to \Gamma(\D_1,\forms[1]_{\D_1})$ is an isomorphism, and since $\X$ is Fano, we have $\Gamma(\X,\forms[1]_\X) = 0$.  Therefore $\D_1$ carries no nonzero holomorphic forms, which implies that $\alpha$ must have poles.  By the residue theorem, the polar divisor of $\alpha$ must have at least two components, and hence by the connectedness of intersections, we conclude that there must be at least two more components $\D_2,\D_3$ such that $\D_1\cap\D_2\cap\D_3 \ne 0$.  Moreover, the biresidue of $\omega$ along $\W = \D_1\cap \D_2$ is nonzero.

It follows, in particular, that the proposition holds for $\dim \X = 4$.  We now proceed by induction.  Suppose that $\dim \X > 4$, and consider the intersection $\W = \D_1 \cap \D_2$ we found in the previous paragraph.  By \autoref{lem:bires}, the Poisson structure on $\W$ is generically symplectic.   Since $\D_3 \cap \W \ne 0$, the intersection $\D_\W = \W \cap (\D - \D_1 - \D_2)$ is a non-empty simple normal crossings divisor on $\W$.  It is a Poisson subspace, and by the adjunction formula, it is also an anticanonical divisor. Hence it must be the degeneracy divisor of $\pi|_\W$.  Applying the Lefschetz hyperplane theorem again, we see that the restriction map $\cohlgy[2]{\X,\ZZ} \to \cohlgy[2]{\W,\ZZ}$ is an isomorphism, and we conclude that $\W$ is a Fano manifold with second Betti number $b_2(\W) = 1$.  Therefore $(\W,\pi|_\W)$ satisfies the hypotheses of the proposition.  By induction, we conclude that $\D_\W$ has at least $\dim \W - 1$ irreducible components.  Using the connectedness of intersections again, we conclude that $\D$ must have at least $\dim \W - 1$ components  in addition to  $\D_1$ and $\D_2$.  Therefore $\D$ has at least $\dim \W - 1 + 2 = \dim \X - 1$ components, as desired.
\end{proof}

To complete the proof of the \autoref{thm:onlyPn}, we recall some basic concepts from the classification of Fano manifolds.  The \defn{index $i(\X)$} of a Fano manifold $\X$ is the largest integer that divides $c_1(\X)$ in $\cohlgy[2]{\X,\ZZ}$.  When $b_2(\X)=1$, we simply have $c_1(\X) = i(\X)H$ where $H \in \cohlgy[2]{\X,\ZZ}$ is a generator.  If $\D$ is an effective divisor, then each component $\D_i$ of $\D$ has fundamental class $[\D_i] = d_iH$ for some $d_i \in \ZZ_{>0}$.  Hence if  $b_2(\X) =1$ and $\D$ is an anticanonical divisor, the number of components of $\D$ gives a lower bound on the index. 

\begin{proof}[Proof of \autoref{thm:onlyPn}]
From \autoref{prop:many-cpts}, we see that $i(\X)$ is at least $n-1$.  Hence by results of Kobayashi--Ochiai~\cite{Kobayashi1973} and Fujita~\cite{Fujita1982,Fujita1982a} (see also \cite{Iskovskikh1999}), we are in one of the following cases:
\begin{enumerate}
\item $i(\X) = n +1$ and $\X \cong \PP^{n}$
\item $i(\X) = n$ and $\X\subset \PP^{n+1}$ is a smooth quadric hypersurface
\item $i(\X) = n-1$ and $\X$ is one of the following manifolds:
\begin{enumerate}
\item a degree-six hypersurface in the $(n+1)$-dimensional weighted projective space $\PP(3,2,1,\ldots,1)$;
\item a double cover of $\PP^n$ branched over a smooth quartic hypersurface;
\item a smooth cubic hypersurface in $\PP^{n+1}$;
\item an intersection of two smooth quadric hypersurfaces in $\PP^{n+2}$
\item a linear section of the Grassmannian $\Gr{2,5}$ in its Plucker embedding.
\end{enumerate}
\end{enumerate}
We dealt with the case $\X = \PP^n$ already in \autoref{prop:toricPn}, so we simply have to rule out the other possibilities, which we can easily do using our constraint on the Chern character (\autoref{prop:chern-sinh}).  We shall describe the case $i(\X)=n$ so that $\X \subset \PP^{n+1}$ is a smooth quadric hypersurface; the remaining cases are treated in exactly the same manner. 

Suppose to the contrary that the hyperquadric $\X$ carries a log symplectic structure with simple normal crossings degeneracy divisor $\D = \D_1+\cdots+\D_k$.  Let $H \in \cohlgy[2]{\X,\ZZ}$ be the restriction of the hyperplane class.  Using the Chern character of $\PP^{n+1}$ and the exact sequence for the normal bundle, we get the equality
\[
\ch(\tshf{\X}) = (n+2)e^H-1-e^{2H}.
\]
Writing $[\D_i] = d_iH$ for $1 \le i \le k$, the formula \eqref{eqn:chern-sinh} for the Chern character gives the equations
\begin{align*}
(n+2)-2 &= d_1+\cdots + d_k \\
(n+2)-2^3 &= d_1^3+\cdots+ d_k^3\\
&\ \ \vdots
\end{align*}
Evidently the sequence of numbers on the left hand side is decreasing, while the sequence on the right hand side is not.  Thus the system has no solutions, a contradiction.
\end{proof}

\begin{exercise}
Use the formulae for the Chern characters in \cite{Araujo2013a} to complete the proof of the theorem.\qed
\end{exercise}

%
%
%
%
%\section{Log symplectic forms}
%
%In higher dimension, classification becomes much harder, need to find a manageable class to try to attack.  One particularly nice class is manifolds that are very close to being symplectic.
%
%Let $\X$ be a complex manifold of dimension $2n$.  Thus it makes sense to ask if a Poisson structure $\pi$ on $\X$ comes from a symplectic form.  this is the case precisely when $\pi^n \in \acan_\X$ is a non-vanishing section of the anticanonical line bundle.  Let us weaken this assumption and suppose that 
%
%is \defn{log symplectic} if $\pi^{n}$ is nonzero, and the degeneracy locus
%
%
%
%A Poisson structure is log symplectic if ...
%
%Have been studied quite intensely in the past while, particularly in the $C^\infty$ setting where they are also known as topologically stable or $b$-symplectic.  Typically it is assumed that the degeneracy locus is a smooth hypersurface, but in many of the natural holomorphic examples, this restriction is too strong.
%
%My purpose in this lecture is to explain how to deal with the singularities, using the theory of logarithmic differential forms.
%
%\subsection{Log forms}
%
%\subsection{Normal forms via the local Moser trick}
%
%\subsection{Classification results for $\PP^4$}
%
%\subsection{Open problems and conjectures}
%
%Fano fourfolds, Lie theoretic examples, classification of reduced Frobenius Lie algebras, finite-dimensionality of cohomology
%

\bibliographystyle{hyperamsplain}
\bibliography{../../../Research/references/master}

\end{document}